\documentclass[a4,10pt]{article}
\usepackage{amsmath,amsthm,amssymb}
\usepackage[dvipdfmx]{graphicx}
\usepackage{fancybox}
\usepackage{color}
\usepackage{tikz}
\usepackage{subfig}

\newtheorem{Th}{Theorem}[section] 
\newtheorem{Lem}[Th]{Lemma} 
\newtheorem{Prop}[Th]{Proposition} 
 
\newtheorem{Def}[Th]{Definition} 
\newtheorem{Prob}[Th]{Problem}

\def\N{{\mathbb N}}
\def\R{{\mathbb R}}

\makeatletter
\def\argmin{\mathop{\operator@font argmin}}
\makeatother

\begin{document}







\title{Particle dynamics subject to impenetrable boundaries: existence and uniqueness of mild solutions}

\author{M.~Kimura \and P.~van Meurs \and Z.X.~Yang}

\date{}

\maketitle
\tableofcontents

\begin{abstract}
We consider the dynamics of particle systems where the particles are confined by impenetrable barriers to a bounded, possibly non-convex domain $\Omega$. When particles hit the boundary, we consider an instant change in velocity, which turns the systems describing the particle dynamics into an ODE with discontinuous right-hand side. Other than the typical approach to analyse such a system by using weak solutions to ODEs with multi-valued right-hand sides (i.e., applying the theory introduced by Filippov in 1988), we establish the existence of mild solutions instead. This solution concept is easier to work with than weak solutions; e.g., proving uniqueness of mild solutions is straight-forward, and mild solutions provide a solid structure for proving many-particle limits.

We supplement our theory of mild solutions with an application to gradient flows of interacting particle energies with a singular interaction potential, and illustrate its features by means of numerical simulations on various choices for the (non-convex) domain $\Omega$.
\end{abstract}

\section{Introduction}\label{sec1}
\setcounter{equation}{0}

In various case studies in the field of interacting particle systems, in particular SPH \cite{gingold1977smoothed,li2007meshfree,liu2010smoothed,lucy1977numerical,monaghan1992smoothed,asai2012stabilized}, pedestrian dynamics \cite{MauryVenel11,bellomo2013microscale,daamen2010capacity,helbing2003lattice,liao2014experimental} and granular media \cite{Stewart98,Maury06}, the particles are confined to a bounded, non-convex domain $\Omega \subset \R^m$ (see Figure \ref{fig:intro}). While there are various ways to implement such impenetrable boundaries in numerical schemes for computing particle dynamics, a satisfactory analytical framework seems to be largely missing. Such a missing structure makes, for example, any statement on many-particle limits out of reach. Our aim is therefore to set up an analytical framework for interacting particle systems which can handle impenetrable barriers, to define a satisfactory notion of a solution for the particle trajectories, and to prove existence and uniqueness of such solutions.
\begin{figure}[htbp]
  \begin{center}
    \includegraphics[width=80mm]{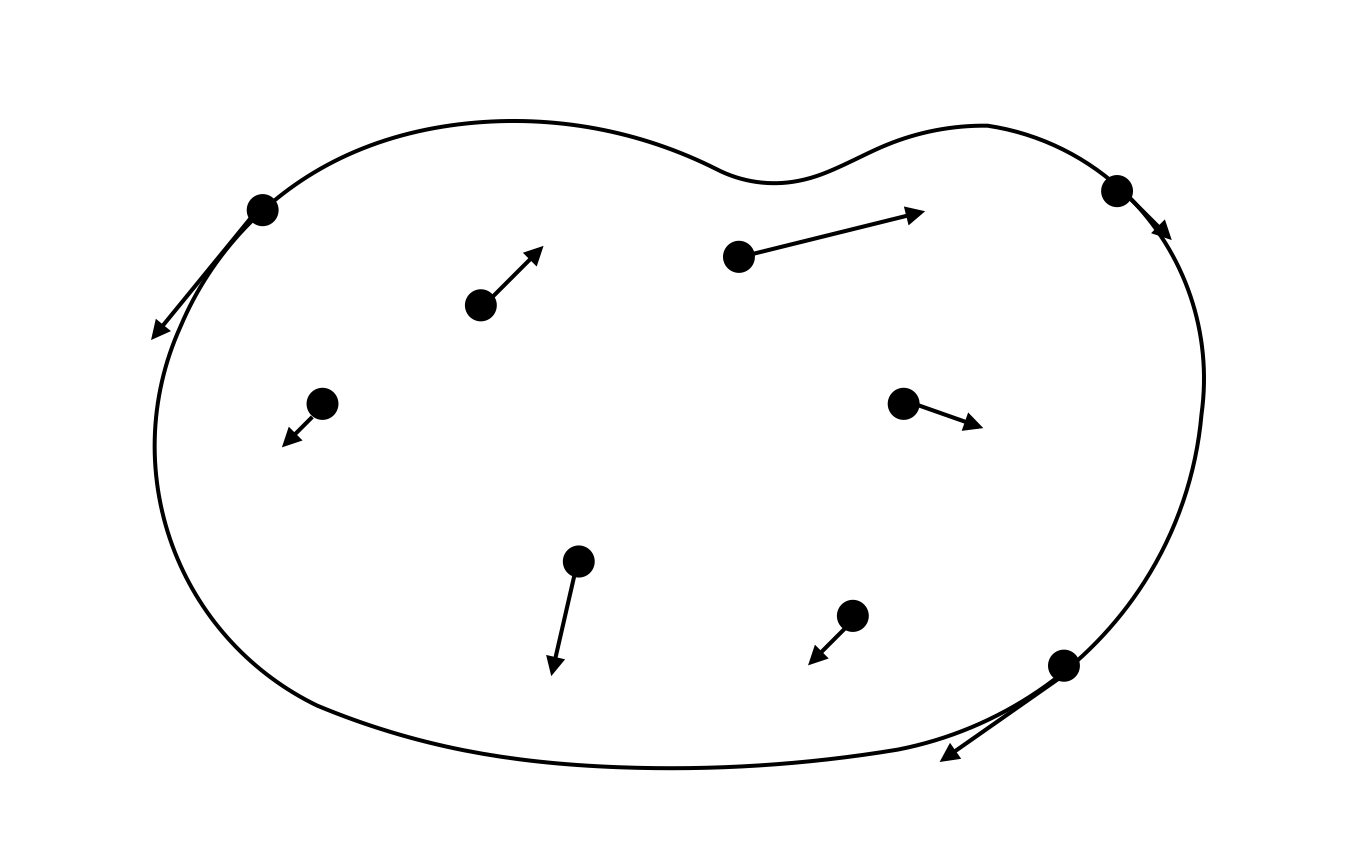}
  \end{center}
  \caption{Particles $x_i$ confined in a domain $\Omega$.}
  \label{fig:intro}
\end{figure}

In this paper, we focus on first-order in time particle systems. To motivate formally and state the related system of ODEs, we first consider the motion of a point mass $x(t)$ with mass $m_0 > 0$
confined to $\Omega$.
We subject the point mass to a given force $f(t)\in\R^m$ and impose a
viscous resistance represented by a linear drag force
$-\mu\dot{x}$ with
a positive coefficient $\mu >0$.
Then, using Newton's second law, the equation of motion for the point mass is
\[
m_0\ddot{x}(t) + \mu \dot{x}(t) = f(t)
\]
whenever $x(t)\in \Omega$. Next we assume that $m_0/\mu \ll 1$, and consider the overdamped limit.
Setting $m_0=0$ and $F(t):=f(t)/\mu$ (we keep on calling it a ``force'' as 
in the above context), we obtain for the dynamics of the point mass that $\dot{x}(t) = F(t)$. Finally we motivate and include the effect of the confining domain $\Omega$. When $x(t)\in \partial \Omega$,
from Newton's third law, a counter force $-(F\cdot\nu)_+\,\nu$ acts on the particle from $\partial \Omega$,
where $\nu$ is the outward normal vector on $\partial \Omega$
and $(a)_+:=\max (a,0)$. This leads to
the following ODE with a ``one-sided projection'':
\[
\dot{x}(t)=\left\{
\begin{array}{ll}
F(t)&\mbox{if $x(t)\in\Omega$,}\\[5pt]
F(t)-(F(t)\cdot\nu)_+\,\nu&
\mbox{if $x(t)\in\partial \Omega$.}
\end{array}
\right.
\]

Using the overdamped dynamics of the point mass, we introduce the dynamics of the particle system. Denoting the positions of $n$ particles by $X(t) = \{x_i(t)\}_{t=1}^n \in \overline \Omega^n$, we consider a given force $F_i \big(t,X(t)\big)$ which acts on particle $i$ for each $i = 1,\ldots,n$. With the confinement described above, this yields the following set of ODEs:
\begin{equation} \label{f1}
\dot{x}_i(t) = \left\{\begin{aligned}
  &F_i \big(t,X(t)\big) & x_i(t) &\in \Omega,\\
  &F_i \big(t,X(t)\big) - \big( F_i \big(t,X(t)\big) \cdot \nu(x_i(t)) \big)_+\,\nu(x_i(t))
  & x_i(t) &\in \partial \Omega.
  \end{aligned}
\right.
\end{equation}
Because of the one-sided projection of the force at the boundary, the trajectories of particles typically have a discontinuity in their velocity at the time when they hit the boundary. Since classical $C^1$-solutions to \eqref{f1} cannot describe such behavior, we cannot rely on classical ODE theory such as the Cauchy-Peano Theorem and the Picard-Lindel$\ddot{o}$f (Cauchy-Lipschitz) Theorem \cite{hartmanordinary} for establishing the existence and uniqueness of solutions to \eqref{f1}. \medskip
 
Next we briefly review the literature on the existence and uniqueness of solutions to systems with a discontinuous right-hand side such as \eqref{f1}. When $\Omega$ is convex and $F_i$ is the gradient of a $\lambda$-convex energy with respect to $x_i$, then gradient flow theory \cite{AmbrosioGigliSavare08} provides a sufficiently strong analytical framework to describe \eqref{f1}. In the general case ($\Omega$ not convex or a more general discontinuous right-hand side than the one in \eqref{f1}), the book by Filippov [Fil88] treats such discontinuities by making the right-hand side \emph{multi-valued}. For the related differential \emph{inclusion}, the existence and uniqueness of weak solutions is sought. Especially the conditions for the right-hand side to result in uniqueness of solutions can be difficult to check, and need to be considered differently on a case-to-case basis. Moreover, weak solutions to differential inclusions seem to lend themselves poorly for the purpose of studying many-particle limits.

As opposed to the general setting in [Fil88], \eqref{f1} has a particular type of discontinuity in the right-hand side which, under appropriate regularity conditions of $F_i$ and $\Omega$, allows for a stronger sense of solutions. Indeed, our main result (Theorem \ref{theorem1}) states that for enough regular $F_i$ and $\Omega$, \eqref{f1} attains a unique, \emph{mild} solution, without the need to transform \eqref{f1} into a differential inclusion. In addition, we show that when particles detach from the boundary, their velocity is tangent to $\partial \Omega$ (Propsition~\ref{property}). \medskip

Regarding the regularity condition on $F$, we require $F$ to be continuous for the existence of mild solutions. However, in many applications of interacting particle systems, $F$ blows up whenever two particles tend to the same position. To cover such scenarios, we extend Theorem \ref{theorem1} to the case in which $F$ is the gradient of an energy which blows up to $+\infty$ as any two particles tend to the same position (as is the case, for instance, for charged particles). This extension is our second main result, which is stated in Theorem \ref{Th3}. 

The intended purpose of Theorem \ref{Th3} is to be a stepping stone for proving the limit passage $n \to \infty$ (i.e., the many-particle limit) from \eqref{f1} to a nonlocal and nonlinear equation for the particle density. The singularity of the particle interaction potential $V$ at $0$ plays a crucial role here, since it has a regularising effect on the particle density. In recent studies, such limit passages $n \to \infty$ have been carried out for:
\begin{itemize}
  \item convex $V$ (except at $0$) in one spatial dimension in \cite{VanMeursMuntean14}, 
  \item logarithmic $V$ in two spatial dimensions in \cite{Schochet96}, which was later extended to stronger singularities in \cite{Duerinckx16}, 
  \item sub-Coulombic potentials (i.e., $V(x) \ll |x|^{2-m}$ as $x \to 0$) in arbitrary spatial dimension $m$ in \cite{Hauray09},
  \item multi-species particle systems with regular cross-interactions in \cite{vanMeurs18,vanMeursMorandotti18ArXiv},
  \item multi-species particle systems with singular cross-interactions in \cite{ChapmanXiangZhu15,GarroniVanMeursPeletierScardia18ArXiv},
  \item a rate-independent velocity law in \cite{MoraPeletierScardia17}.
\end{itemize}
Besides singular potentials, it is also possible to obtain a regular particle density by adding white noise to the particle dynamics. It has been proven in \cite{Sznitman91,Philipowski07} that such systems satisfy the propagation of chaos property, which implies that the limit can be described by a deterministic equation for the particle density. However, in all these studies, the problem of particles attaching to or detaching from boundaries is side-stepped. With Theorem \ref{Th3} in hand, we expect that it is possible to extend the techniques in the references cited above to establish the limit passage $n \to \infty$ in \eqref{f1}. We leave this challenge for future research. \medskip

Another interesting generalisation of our result is the extension to second order ODE systems with inertial term. This generalisation has been studied by \cite{schatzman1978class} for convex energies and convex domains. We believe that our current result provides a new viewpoint from which Schatzman's result can be extended to non-convex energies (which allows, e.g., for repulsive particle interactions) and non-convex domains. An important application that we have in mind for this is a properly motivated boundary condition in particle simulations for fluids. Indeed, in many such simulations (in particular when SPH is used), there is no established theory on how to treat properly boundary conditions for the particles which effectively model slip and no-slip conditions in the equations for the fluid.\medskip

To gain insight in the solution trajectories of \eqref{f1} (in particular, how particles attach to and detach from the boundary), we perform several numerical simulations of the solution to \eqref{f1}. We study the case where the force $F_i$ is the gradient of an unbounded interaction energy with an external potential. Our interest lies in the behavior of the system for several different, non-convex domains $\Omega$. \medskip

The paper is organised as follows. In Section \ref{sec2} we state and prove Theorem \ref{theorem1}, which provides the existence and uniqueness of mild solutions to \eqref{f1}. In Section \ref{sec3} we extend Theorem \ref{theorem1} to Theorem \ref{Th3}, which allows for a class of discontinuous $F$. In Section \ref{sec4} we perform and discuss our numerical simulations. We conclude in Section \ref{sec5}. 

\section{Discontinuous ODE system: existence and uniqueness}\label{sec2}
\setcounter{equation}{0}
\subsection{Problem setting and main theorem}
\label{ss:PnT}

Let $\Omega$ be a bounded domain in $\R^m$ with $C^2$ boundary $\partial\Omega$. We denote the outward unit normal vector on $\partial\Omega$ by $\nu$.

We consider $n$ particles $x_1,\dots ,x_n$ in $\overline\Omega$. We put $X:=\{x_i\}_{i=1}^n \in \overline\Omega^n \subset (R^m)^n$ as the array of the particle's positions. On $(\R^m)^n$ we define the norm
\[
\|X\| := \max_{1\le i\le n}|x_i|,
\]
where $|x_i|$ is the Euclidean distance in $\R^m$.
We interpret
\begin{equation}\label{G}
  G := \overline\Omega^n 
\end{equation}
as the set of admissible particle positions.

Next we introduce the dynamics of the particle system. We denote by $X(t) = \{x_i(t)\}_{i=1}^n$ the particle positions in time, and by $\dot{x_i}(t)=\frac{dx_i}{dt}(t)$ their velocity. We prescribe the velocity of any particle $x_i(t)$ depending on whether $x_i(t)$ is in the interior or at the boundary of $\Omega$. If $x_i(t) \in \Omega$, then its velocity is given by the force $F_i(t, X(t))$, where 
\begin{equation}\label{1}
F_i\in C([0,T]\times G;\R)\quad(i=1,\dots,n).
\end{equation}
In some cases, we impose the following hypothesis on Lipschitz continuity of $F_i$: there exists an $L > 0$ such that
\begin{equation}\label{2}
|F_i(t,X)-F_i(t,Y)|\le L\|X-Y\| \quad (t\in [0,T],~X,Y\in (\R^m)^n,~i=1,\dots,n).
\end{equation} 

If $x_i(t) \in \partial \Omega$, then we prescribe a different rule. If the force $F_i(t, X(t))$ acting on $x_i(t)$ is directed to the interior of $\Omega$, then we prescribe its velocity by $F_i(t, X(t))$. Otherwise, we project $F_i(t, X(t))$ onto the tangent plane of $\partial \Omega$ at $x_i(t)$ (see Figure \ref{Omega}). To define this rule mathematically, we put
\begin{equation}\label{P}
P : \overline\Omega \times \R^m \to \R^m, \quad
P(x,f):=\left\{
\begin{aligned}
  &  f - \big(f \cdot \nu(x)\big) \nu(x) &&(x\in\partial\Omega, f\cdot \nu(x)>0),\\
  &  f  &&\text{otherwise},
\end{aligned}
\right.
\end{equation}
and prescribe the velocity of $x_i(t)$ as $P (x_i(t),F_i(t,X) )$. For convenience, we set
\[
H_i : [0,T] \times G \to \R^m, \quad
H_i \big( t,X \big) := P\big(x_i, F_i(t,X)\big)\quad(t\in [0,T], X=\{x_i\}\in G).
\]

\begin{figure}[htbp]
  \begin{center}
    \includegraphics[width=80mm]{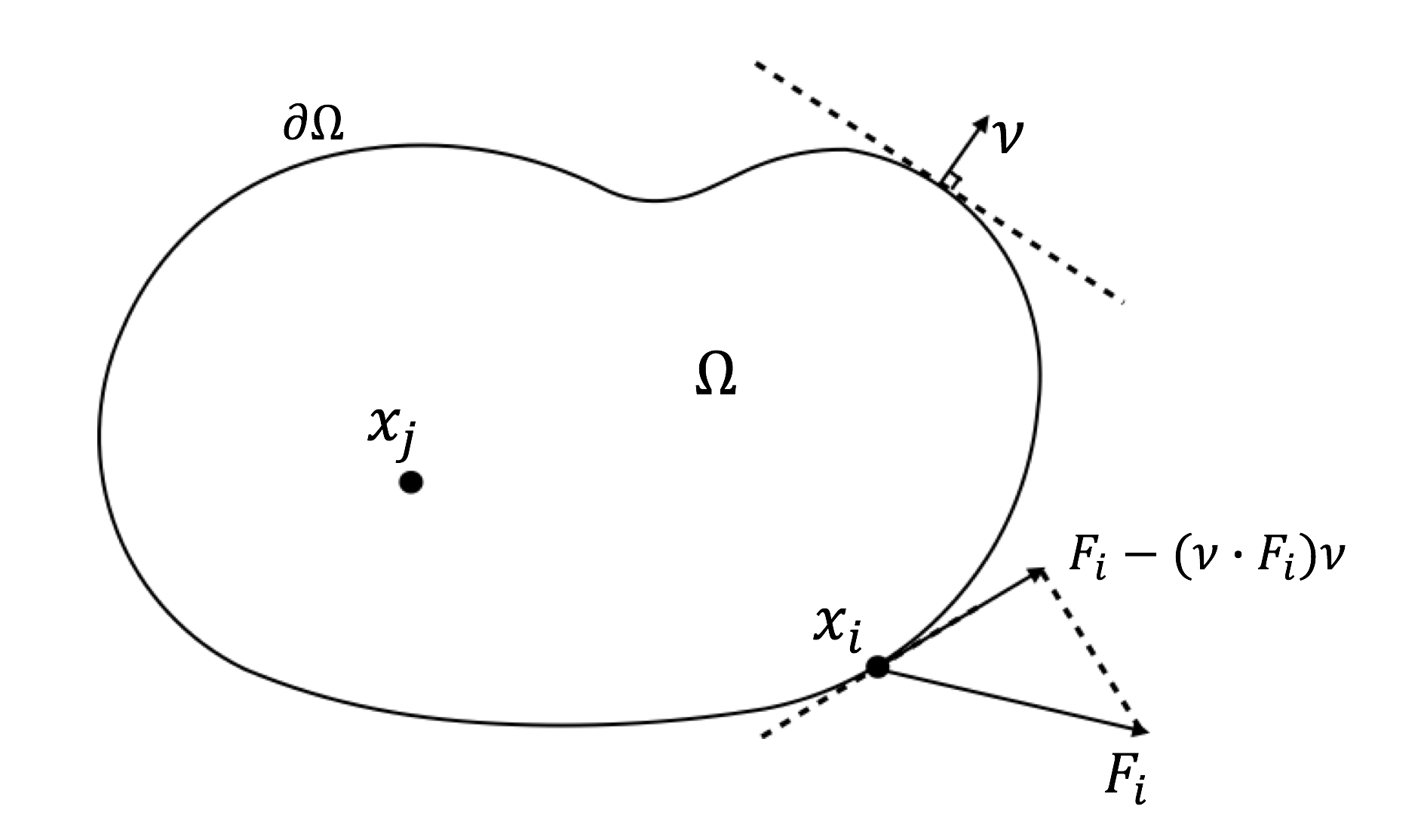}
  \end{center}
  \caption{Sketch of $\overline{\Omega}$ and the role of the 'one-sided projection' of the force $F_i$ to prevent particles from escaping $\overline\Omega$.}
  \label{Omega}
\end{figure}

To summarise the dynamics of the particles, we introduce the following problem:
\begin{Prob}\label{problem1}
Given $T > 0$ and an initial condition $X^\circ =\{x_i^{\circ}\}_{i=1}^n\in G$, for $i=1,2,\dots,n$, find a solution to
\[
\left\{
\begin{aligned}
  & \dot{x}_i(t)=H_i\big(t,X(t)\big)\quad (0 < t < T),\\
  & x_i(0)=x_i^{\circ}.
  \end{aligned}
  \right.
\]
\end{Prob}

Since $H_i$ is discontinuous, we do not expect solutions to Problem~\ref{problem1} to be $C^1$ in time. Instead, we consider mild solutions:
  
\begin{Def}\label{sol}
  We call $X(t)$ a (mild) solution to Problem~\ref{problem1} if it satisfies $x_i\in C^0([0,T];\R^m)$, $x_i(t)\in\overline{\Omega}$ and
  \[
  x_i(t)=x_i^{\circ}+\int_0^tH_i\big(s,X(s)\big)ds\quad \big(t\in [0,T)\big),
  \]
  for $i=1,2,\dots,n$.
\end{Def}

The following proposition gives a convenient characterisation of mild solutions. The proof is obvious.

\begin{Prop} \label{prop:23}
For Problem~\ref{problem1}, $X(t)=\{x_i(t)\}_{i=1}^n$ is a mild solution if and only if
\[
\left\{
\begin{aligned}
  & x_i\in\mathrm{Lip}([0,T];\R^m)\\
  & \dot{x}_i(t)=H_i(t,X(t))\quad \big(\text{a.e.~} t\in[0,T]\big)\\
  & x_i(0)=x_i^{\circ}.
  \end{aligned}
  \right.
\]
\end{Prop}
Our main result is the following theorem, which grants the existence and uniqueness of mild solutions to Problem \ref{problem1}:

\begin{Th}\label{theorem1}
Let $T>0$ and $X^{\circ}\in G$. If $F_i$ satisfies \eqref{1}, then there exists a solution to Problem~\ref{problem1}. Moreover, if $F_i$ satisfies the Lipschitz condition \eqref{2}, then the solution is unique.
\end{Th}
This theorem is a natural extension of Cauchy-Peano theorem and the Picard-Lindel$\ddot{o}$f theorem to particles systems with impenetrable boundaries. The proof for the existence and the uniqueness is given in Sections~\ref{Existence} and \ref{Uniqueness}, respectively.

\subsection{Distance functions with respect to $\Omega$ and $\partial \Omega$}

In the proof of Theorem \ref{theorem1}, we define an approximation problem in which particles can leave $\overline\Omega$. To account for this, we define in this section distance functions on $\R^m$ with respect to $\Omega$ and $\partial\Omega$. For $K\subset\R^m$ and $\varepsilon>0$, we defne
\[
\text{dist}(x,K):=\inf_{y\in K}|x-y|,
\]
\[
N^{\varepsilon}(K):=\{x\in\R^m~;~\text{dist}(x,K)<\varepsilon\}.
\]
  For the domain $\Omega$, we define the distance function $d$ with respect to $\Omega$ and the signed distance function $d_s$ with respect to $\partial \Omega$ as
  \[
  \begin{aligned}
    & d(x):=\mathrm{dist}(x,\overline{\Omega})\quad(x\in\R^m),\\
    & d_s(x):=\left\{
    \begin{aligned}
      & \quad \mathrm{dist}(x,\partial\Omega)\quad(x\notin\overline{\Omega}),\\
      & -\mathrm{dist}(x,\partial\Omega)\quad(x\in\overline{\Omega}).
    \end{aligned}
    \right.
  \end{aligned}
  \]
  
Since $\partial\Omega\in C^2$, there exists $\varepsilon>0$ such that $d\in C^2(\overline{N^{\varepsilon}(\partial\Omega)}\setminus\Omega)$ and $d_s\in C^2(\overline{N^{\varepsilon}(\partial\Omega)})$. We fix this $\varepsilon$ throughout Section~\ref{Existence} and Section~\ref{Uniqueness}. It is known (see, e.g., \cite{gilbarg2015elliptic,Kimura2008}) that the signed distance function $d_s$ satisfies
\[
\nabla d_s(x)=\nu(\bar{x}),~\bar{x}=x-d_s(x)\nabla d_s(x)\in \partial\Omega\quad\big(x\in\overline{N^{\varepsilon}(\partial\Omega)}\big),
\]
where $\bar{x}\in\partial\Omega$ denotes the orthogonal projection from $x\in\overline{N^{\varepsilon}(\partial\Omega)}$ onto $\partial\Omega$. We set
\begin{equation}\label{df}
d\nabla d(x):=\left\{
\begin{aligned}
& d(x)\nabla d(x)
& (x\in\overline{N^{\varepsilon}(\partial\Omega)}\setminus\overline{\Omega}),\\
& 0
& (x\in\overline{\Omega}).
\end{aligned}
\right.
\end{equation}

Let $\varepsilon>0$ as above. We put $d_s^{\varepsilon}\in C^2(\R^m)$ as a smooth extension of $d_s$ outside $N^{\varepsilon}(\partial \Omega)$:
  \[
  d_s^{\varepsilon}(x) \left\{
  \begin{aligned}
    & =d_s(x)
    & \big(x &\in N^{\varepsilon}(\partial \Omega)\big), \\
    & \ge \varepsilon 
    & \big(x &\in\R^m\setminus N^{\varepsilon}(\Omega)\big),\\
    & \le -\varepsilon 
    & \big(x &\in\Omega\setminus N^{\varepsilon}(\partial \Omega)\big).
  \end{aligned}
  \right.
  \]
We note that, by construction, $\nabla d_s^{\varepsilon}$ is a $C^1$-extension of $\nu$ from $\partial \Omega$ to $\R^m$. 
Finally, we establish the following Lipschitz property of $d\nabla d$:
\begin{Prop}\label{d}
The function $d\nabla d(x)$ defined in \eqref{df} satisfies
  \[
  d\nabla d\in \mathrm{Lip}(\overline{N^{\varepsilon}(\Omega)};\R^m).
  \]
\end{Prop}

\begin{proof}
We show that there exists a constant $C' > 0$ such that for all $x,~y\in \overline{N^{\varepsilon}(\Omega)}$,
\begin{equation}\label{pf:p1}
|d\nabla d(x)-d\nabla d(y)|\le C' |x-y|   
\end{equation}
holds. We separate three cases.
\smallskip

Case 1: $x,~y\in\overline{\Omega}$. Since $d\nabla d \equiv 0$ on $\overline{\Omega}$, \eqref{pf:p1} is obvious.
\smallskip

Case 2: $x,~y\in \overline{N^{\varepsilon}(\Omega)}\setminus\overline{\Omega}$.
Since $d\nabla d(x)\in C^1(N^{\varepsilon}(\partial\Omega)\setminus\overline{\Omega})$, \eqref{pf:p1} is obvious.
\smallskip

Case 3: $x\in\overline{\Omega}$ and $y\in \overline{N^{\varepsilon}(\Omega)}\setminus\overline{\Omega}$. We conclude \eqref{pf:p1} from
\[
|d\nabla d(x)-d\nabla d(y)|=|d(y)\nabla d(y)|=d(y)=\min_{z\in\overline{\Omega}}|z-y|\le |x-y|.
\]
\end{proof}


\subsection{Preliminaries and notation}
\label{ss:PnN}

In this section, we introduce some notation for the proof of Theorem \ref{theorem1}, and recall some well-known theorems.

Using the assumed regularity on $\Omega$ and $F_i$, we extend $F_i$ to $[0, \infty) \times (\R^m)^n$ such that $F_i$ is still continuous. In the remainder of the paper, we write $F_i$ for this extension.

We define the $L^{\infty}$-norm of $F := (F_1, \ldots, F_n)$ as
  \[
  \|F\|_{\infty} := \max_{t \in [0,T]}\max_{X\in (\R^m)^n}\max_{1\le i\le n}|F_i(t,X)|.
  \]
We set $L^{\infty}(0,T;[0,1]):=\{\chi\in L^{\infty}(0,T)~;~\chi(t)\in [0,1]~\text{for a.e.}~t\}$, and denote the positive and negative part of a number or function by
  \[
  (a)_+:=\max(0,a),~(a)_-:=\max(0,-a).
  \]

As preparation for the proof of Theorem \ref{theorem1}, we recall two well-known results.
\begin{Th}[Ascoli-Arzela] \label{AA} 
  Let $(\mathcal X, \mathsf m)$ be a compact metric space. A sequence $\{f_n\}\in C(\mathcal X)$ has a uniformly convergent subsequence if it is bounded and equicontinuous, i.e., $\sup_n\|f_n\|<\infty$ and for every $\varepsilon>0$ there exists $\delta>0$ such that for all $n$ and all $x,~y\in \mathcal X$ with $\mathsf m(x, y) < \delta$, it holds that $|f_n(x)-f_n(y)|<\varepsilon$.
\end{Th}

\begin{Th}[Co-area formula on $\R$ \cite{evans2018measure}]\label{CA} 
  Let $A \subset \R$ be a Borel set, and $u : A \to \R$ be Lipschitz continuous. Then, $y \mapsto \mathcal{H}^0\big(u^{-1}(y)\big)$ is measurable, and
  \[
  \int_{A}|u'(t)| dt =\int_{\R} \mathcal{H}^0\big(u^{-1}(y)\big) dy,
  \]
  where $\mathcal{H}^0$ is the zero dimensional Hausdorff measure (counting measure) on $\R$.
\end{Th}

\subsection{The approximate problem to Problem \ref{problem1}}
\label{apx:prob}

In preparation for proving the existence of solutions to Problem \ref{problem1}, we define Problem \ref{problem2} to construct a sequence $\{X^k(t)\}_k$ of approximate solutions. We recall that we extended $F_i$ from $[0,T] \times G$ to $[0,\infty) \times (\R^m)^n$ in Section \ref{ss:PnN}. 

\begin{Prob}\label{problem2}
  Given $k\in\N$ and $X^\circ \in G$, find $X^k(t)=\{x_i^k(t)\}_{i=1}^n \in (\R^m)^n$ such that
  \begin{equation}\label{dotx}
  \left\{
  \begin{aligned}
    & \dot{x}_i^k(t)=F_i(t,X^k(t))-k(d\nabla d)(x_i^k(t)) \quad (0 < t < T)\\
    & x_i^k(0)=x_i^{\circ}. 
    \end{aligned}
  \right.
  \end{equation}
  \end{Prob}

Figure \ref{example} illustrates a typical situation for a solution to Problem \ref{problem2}. The following lemma describes properties of solutions to Problem \ref{problem2}.  
 \begin{figure}[htbp]
      \begin{center}
        \includegraphics[width=80mm]{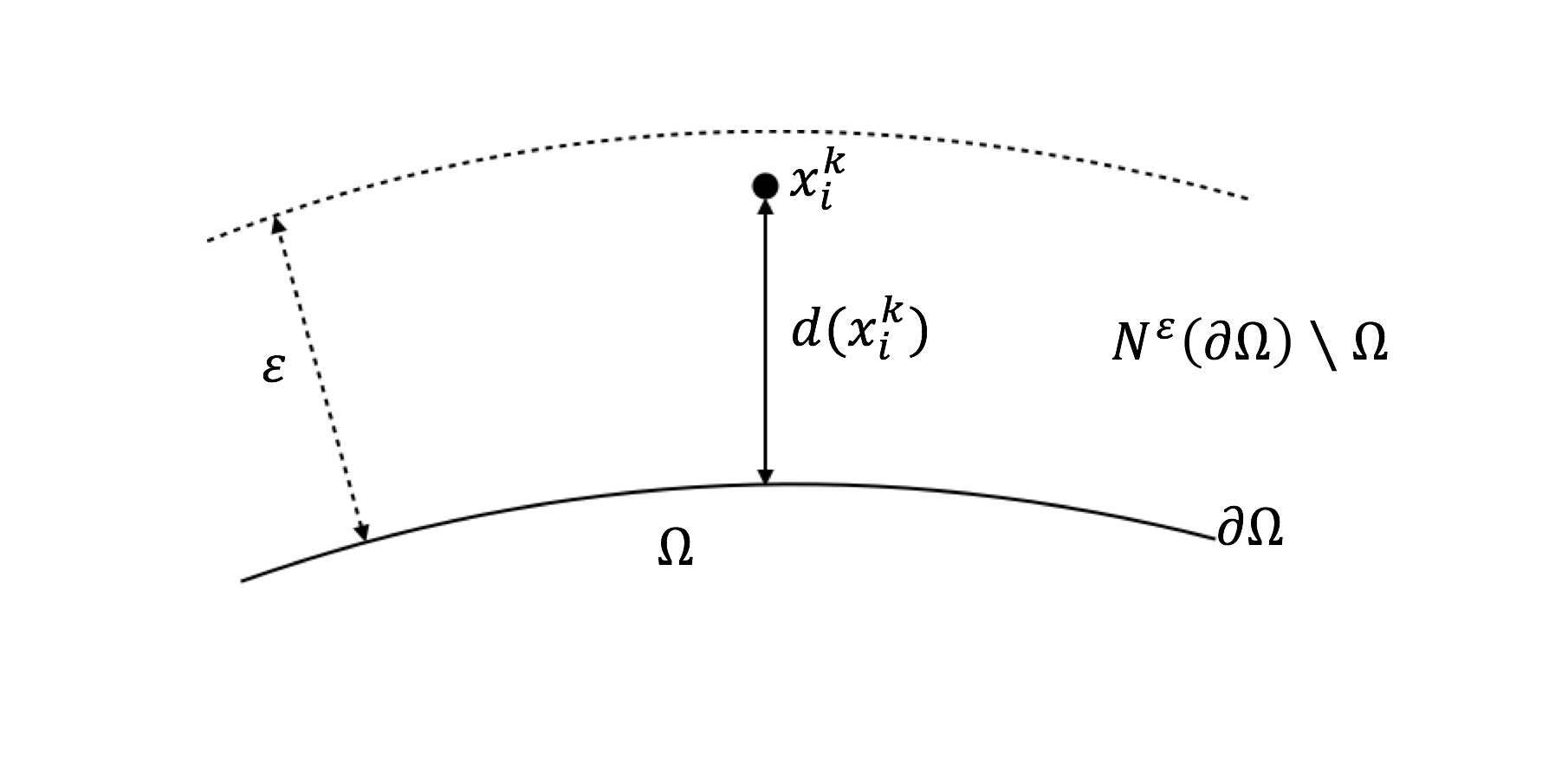}
      \end{center}
      \caption{A situation sketch for particle $x_i^k(t)$ of a possible solution to Problem \ref{problem2}.}
      \label{example}
 \end{figure}
 
 \begin{Lem}\label{Lemma1}
Let $T > 0$ and suppose that \eqref{1} holds and that $k>\frac{1}{\varepsilon}\|F\|_{\infty}$. Then for any $X^\circ \in G$, there exists a classical solution to Problem~\ref{problem2}. Furthermore, any solution $X^k(t)=\{x_i^k(t)\}$ to Problem~\ref{problem2} satisfies $x_i^k(t)\in N^{\varepsilon}(\Omega)$ and $|\dot{x}_i^k(t)|\le 2\|F\|_{\infty}$ for all $i=1,2,\dots,n$ and all $t \in [0, T]$.
 \end{Lem}

 \begin{proof}
 Since $F_i$ and $d\nabla d$ are continuous in $\overline{N^{\varepsilon}(\Omega)}$ and $X^\circ \in N^{\varepsilon}(\Omega)$, it follows from the Cauchy-Peano theorem that there exists $T_0>0$ for which Problem~\ref{problem2} has a  classical solution $x_i^k(t)$ on $[0,T_0]$ such that $x_i^k(t)\in N^{\varepsilon}(\Omega)$ for all $0 \leq t < T_0$ and such that either $x_i^k(T_0)\in\partial N^{\varepsilon}(\Omega)$ for some $i\in \{1,2,\dots ,n\}$ or $T_0>T$. 
 
Next, we claim that $x_i^k(T_0)\in\partial N^{\varepsilon}(\Omega)$ is impossible, from which it follows that $T_0 > T$. To prove this claim, we suppose that $x_i^k(t)\in N^{\varepsilon}(\Omega)\setminus\overline{\Omega}~(0\le t_0<t<t_1\le T_0)$ and $x_i^k(t_0)\in\partial\Omega$ hold for some $i\in\{1,2,\dots,n\}$ and $[t_0,t_1]\subset [0,T_0]$. We first compute, for any $t\in (t_0,t_1)$,
  \begin{equation}
  \begin{aligned}\label{ddf}
    \frac{d}{dt}d(x_i^k(t))&= \nabla d(x_i^k(t))\cdot \big(F_i(t,X^k(t))-kd(x_i^k(t))\nabla d(x_i^k(t))\big)\\
    &=\nabla d(x_i^k(t))\cdot F_i(t,X^k(t))-kd(x_i^k(t)),
    \end{aligned}
  \end{equation}
and
  \[
  \begin{aligned}
    \frac{d}{dt}\big(e^{kt}d(x_i^k(t))\big) &=e^{kt}\Big(\frac{d}{dt}d(x_i^k(t))+kd(x_i^k(t))\Big)\\
    &=e^{kt}\big(\nabla d(x_i^k(t))\cdot F_i(t,X^k(t))\big)\\
    & \le |F_i(t,X^k(t))|e^{kt}\\
    & \le \|F\|_{\infty}e^{kt}.
    \end{aligned}
  \]
Integrating the above estimate from $t_0$ to $t\in (t_0,t_1]$, we obtain
  \[
  e^{kt}d(x_i^k(t))
  \le e^{kt_0}d(x_i^k(t_0))+\|F\|_{\infty}\int_{t_0}^te^{ks}ds
  = \|F\|_{\infty}\frac{e^{kt}-e^{kt_0}}{k},
  \]
  and thus
  \[
  d(x_i^k(t))\le \|F\|_{\infty}\frac{1-e^{-k(t-t_0)}}{k}\le \frac{\|F\|_\infty}{k} < \varepsilon.
  \]
  Hence, $x_i^k(t) \in N^{\varepsilon}(\Omega)$ for all $i\in \{1,2,\dots,n\}$ and $t\in [0,T_0]$.
  
Finally, we prove the bound on the velocity. By the computation above, $x_i^k(t)\in N^{\varepsilon}(\Omega)$ for all $t \in [0,T]$, and
  \[
  \begin{aligned}
    |\dot{x}_i^k(t)|&=|F_i(t,X^k(t))-k(d\nabla d)(x_i^k(t))|
    \le \|F\|_{\infty}+ k\frac{\|F\|_{\infty}}{k} \big| \nabla d (x_i^k(t)) \big|
    =2\|F\|_{\infty}.
    \end{aligned}
  \]
  \end{proof}


\subsection{Existence of solutions}
\label{Existence}

In this section, we prove the existence statement of Theorem \ref{theorem1}. Let $T > 0$ and the initial condition $X^{\circ}=\{x_i^{\circ}\}_{i=1}^n \in G$ be given. Let $k_0 \in \N$ with $k_0 > \frac1\varepsilon \|F\|_{\infty}$, and $X^k (t)$ be the corresponding solution to Problem \ref{problem2} for all $k \geq k_0$. We use limit points of the sequence $X^k$ as candidates for solutions of Problem \ref{problem1}. With this aim, we first specify the topology in which we seek such limit points (Lemma \ref{lemmaX}).

\begin{Lem}\label{lemmaX}
 There exists a subsequence of $\{X^k\}_{k=k_0}^{\infty} \subset C^1([0,T]; (\R^m)^n)$~(not relabelled) and $X=\{x_i(\cdot)\}_{i=1}^n\in W^{1,\infty}(0,T;(\R^m)^n)\cap C^0([0,T];G)$ such that
  \begin{equation}\label{weak}
  \left\{
  \begin{aligned}
    & x_i^k\to x_i\qquad\text{strongly in}~C^0([0,T];\R^m),\\
    & \dot{x}_i^k\rightharpoonup \dot{x}_i\qquad\text{weakly-$\ast$ in}~L^{\infty}(0,T;\R^m),
    \end{aligned}
  \right.
  \end{equation}
  for all $i=1,\dots,n$ as $k\to\infty$.
\end{Lem}

\begin{proof}
  From Lemma~\ref{Lemma1} we observe that $\{X^k\}_k$ is a bounded sequence in $C^0([0,T];\R^m)$ and equicontinuous on $[0,T]$. From Ascoli-Arzela's Theorem (Theorem~\ref{AA}), there exists a uniformly convergent subsequence $X^{k_l}$ of $X^k$. We define $X$ as its limit as $l \to \infty$, which completes the proof of the first of the two convergence statements in \eqref{weak}. Next we fix $i \in \{1, \ldots, n\}$. From the uniform bound on $\max_{0 \leq t \leq T} |\dot{x}_i^k(t)|$, which is guaranteed by Lemma~\ref{Lemma1}, there exists a subsequence of $\{X^{k_l}(t)\}_l$ (not relabeled) and $y_i\in L^{\infty}(0,T;\R^m)$ such that $\dot{x}_i^{k_l}\rightharpoonup y_i$ in $L^{\infty}(0,T;\R^m)$ as $l \to \infty$. We characterise $y_i$ as $\dot{x}_i$ by
  \[
  \begin{aligned}
  \int_0^Ty_i(t)\cdot \psi(t)dt &= \lim_{l\to\infty}\int_0^T \dot{x}_i^{k_l}(t)\cdot\psi(t)dt=-\lim_{l\to\infty}\int_0^T x_i^{k_l}(t)\cdot\psi '(t)dt\\
  &=-\int_0^T x_i(t)\cdot\psi '(t)dt=\int_0^T \dot{x}_i(t)\cdot\psi(t)dt.
  \end{aligned}
  \]
  \end{proof}

Let $X$ be given by Lemma \ref{lemmaX}. In the remainder of the proof we pass to the limit $k \to \infty$ in the weak form of \eqref{dotx} in order to show that $X$ satisfies the second of the three conditions in Proposition \ref{prop:23}. Let $\psi\in C_0^{\infty}((0,T);\R^m)$ be a test function. To simplify notation, we define $F_i^k(t):=F_i(t,X^k(t))$ and $F_i(t):=F_i(t,X(t))$. We fix $i\in \{1,\dots,n\}$ and define the open set $I_k:=\{t\in[0,T]~;~x_i^k(t)\notin \overline{\Omega}\}$. Starting from the weak form of \eqref{dotx}, we use \eqref{ddf} to obtain
  \begin{equation}\label{ss}
  \begin{aligned}
   \int_0^T(\dot{x}_i^k(t)-F_i^k(t))\cdot\psi(t)dt &=-\int_0^T k(d\nabla d)(x_i^k(t))\cdot \psi(t)dt\\
   & =-\int_{I_k} kd(x_i^k(t))\nabla d(x_i^k)\cdot\psi(t)dt\\
   & =\int_{I_k}\big(\frac{d}{dt}d(x_i^k(t))-\nabla d(x_i^k(t))\cdot F_i^k(t)\big)\nabla d(x_i^k)\cdot\psi(t)dt\\
   & = A_k-B_k,
    \end{aligned}
  \end{equation}
  where  
  \[
  \begin{aligned}
  & A_k=\int_{I_k}\big(\frac{d}{dt}d(x_i^k(t))\big)\nabla d(x_i^k(t))\cdot\psi(t)dt,\\
  & B_k=\int_{I_k}\big(\nabla d(x_i^k(t))\cdot F_i^k(t)\big)\big(\nabla d(x_i^k(t))\cdot \psi(t)\big)dt.
  \end{aligned}
  \]
  
Using Lemma \ref{lemmaX} it is easy to pass to the limit $k \to \infty$ in the left-hand side of \eqref{ss}. Passing to the limit in the right hand side requires more care. We start by proving that
\begin{equation} \label{lemmaA}
  \lim_{k\to\infty}A_k=0.
\end{equation}
With this aim, we decompose $I_k=\bigcup_{\lambda=1}^{\infty}Q_{\lambda}$, where $Q_{\lambda}$ are open connected components of $I_k$. For any $\lambda$, we calculate
\[
  \begin{aligned}
    \left|\int_{Q_{\lambda}}\left(\frac{d}{dt}d(x_i^k(t))\right)\nabla d(x_i^k(t))\cdot\psi(t) dt\right|&=\left|-\int_{Q_{\lambda}}d(x_i^k(t))\frac{d}{dt}\big(\nabla d(x_i^k(t))\cdot\psi(t)\big)dt\right|\\
    &=\left|-\int_{Q_{\lambda}}\underbrace{d(x_i^k(t))}_{\le\frac{\|F\|_{\infty}}{k}}\underbrace{[\big((\nabla^2d(x_i^k(t))\dot{x}_i^k(t)\big)\cdot \psi(t)+\nabla d(x_i^k(t))\cdot\dot{\psi}(t)]}_{\le C}dt\right|\\
    &\le |Q_{\lambda}|\frac{C\|F\|_{\infty}}{k}.
    \end{aligned}
\]
Since $I_k=\bigcup_{\lambda=1}^{\infty}Q_{\lambda}\subset (0,T)$, we obtain
\[
  |A_k|\le \sum_{\lambda=1}^{\infty} C\|F\|_{\infty}\frac{|Q_{\lambda}|}{k}\le\frac{CT\|F\|_{\infty}}{k}\to 0\quad(k\to\infty),
\]
which proves \eqref{lemmaA}.

Next we prepare for passing to the limit in $B_k$ as $k \to \infty$. We define the indicator functions as
\[
\begin{aligned}
   \chi_k(t) &=
  \left\{
  \begin{aligned}
    & 1\qquad x_i^k(t)\notin\overline{\Omega}\\
    & 0\qquad x_i^k(t)\in\overline{\Omega},
  \end{aligned}
  \right.\\
   \chi_k^*(t) &=
  \left\{
  \begin{aligned}
    & 1\qquad x_i^k(t)\notin\overline{\Omega},~ \nabla d_s(x_i^k(t))\cdot F_i^k(t)>0\\
    & 0\qquad \text{otherwise,}
  \end{aligned}
  \right.\\
  \chi_{\partial\Omega}(t) &=
  \left\{
  \begin{aligned}
    & 1\qquad x_i(t)\in\partial\Omega\\
    & 0\qquad x_i(t)\in\Omega.
  \end{aligned}
  \right.
\end{aligned}
\]
We further set
 \begin{align}\label{g}
 g_k(t) &:= \nabla d_s^{\varepsilon}(x_i^k(t))\cdot F_i^k(t), 
 & h_k(t) &:= \nabla d_s^{\varepsilon}(x_i^k(t))\cdot \psi(t), 
   \\\label{h}
 g(t) &:= \nabla d_s^{\varepsilon}(x_i(t))\cdot F_i(t), 
 & h(t) &:= \nabla d_s^{\varepsilon}(x_i(t))\cdot \psi(t), 
 \end{align}
and observe that
 \[
 \lim_{k\to\infty}\|g_k-g\|_{C^0 ([0,T]) } = 0, \quad \lim_{k\to\infty}\|h_k-h\|_{C^0 ([0,T]) } =0.
 \]

\begin{Lem}\label{case4}
For all $t\in (0,T]$, there holds
$
\chi_{\partial\Omega}(t)\big(g(t)\big)_-=0.
$
\end{Lem}

\begin{proof}
We prove the assertion by the argument of contradiction. Suppose that the assertion is false. Then, there exists $t_0\in (0,T]$ such that $x_i(t_0)\in\partial\Omega$ and $\alpha:=-g(t_0)>0$ hold. By the continuity of $g$, there exists $\delta\in (0,t_0)$ such that $g(t)\le -\frac{\alpha}{2}$ for $t_0-\delta\le t\le t_0$. We choose $k_1\ge k_0$ such that $g_k(t)\le -\frac{\alpha}{4}$ holds for all $k\ge k_1$ and all $t_0-\delta\le t\le t_0$. Since
  \[
  \begin{aligned}
  \frac{d}{dt}d_s^{\varepsilon}(x_i^k(t))&=\nabla d_s^{\varepsilon}(x_i^k(t))\cdot \dot{x}_i^k(t)\\
  &=\nabla d_s^{\varepsilon}(x_i^k(t))\cdot\big(F_i^k(t)-kd(x_i^k(t))\nabla d_s^{\varepsilon}(x_i^k(t))\big)\\
  &=g_k(t)-kd(x_i^k(t)),
  \end{aligned}
  \]
  we get
  \[
  \begin{aligned}
      d_s^{\varepsilon}(x_i^k(t_0))&=d_s^{\varepsilon}(x_i^k(t_0-\delta))+\int_{t_0-\delta}^{t_0}\big(g_k(s)-kd(x_i^k(s))\big)ds\\
      &\le d_s^{\varepsilon}(x_i^k(t_0-\delta))-\frac{\alpha}{4}\delta.
  \end{aligned}
  \]
Passing to the limit $k\to\infty$, we obtain $d_s^{\varepsilon}(x_i(t_0))\le d_s^{\varepsilon}(x_i(t_0-\delta))-\frac{\alpha\delta}{4}\le -\frac{\alpha\delta}{4}$, which contradicts $x_i(t_0)\in\partial\Omega$.
\end{proof}

Using the preparations above, we pass to the limit in $B_k$ as $k \to \infty$ in the next lemma.

\begin{Lem}\label{lemmaB}
 There exists $\overline{\chi}\in L^{\infty}(0,T;[0,1])$ such that
  \[
 \lim_{k\to\infty} B_k=\int_0^T\big(\nu(x_i(t))\cdot F_i(t)\big)_+\big(\nu(x_i(t))\cdot\psi(t)\big)\overline{\chi}(t)\chi_{\partial\Omega}(t)dt 
 \]
 \end{Lem}
 
 \begin{proof}
 We define
 \[
 \rho_k(t):=g_k(t)h_k(t)(\chi_k(t)-\chi_k^*(t)\chi_{\partial\Omega}(t))
 \quad \text{and} \quad 
 C_k:=\int_0^Tg_k(t)h_k(t)\chi_k^*(t)\chi_{\partial\Omega}(t)dt
 \]
 such that we can split $B_k$ as
 \[
 B_k-C_k=\int_0^T\rho_k(t)dt.
 \]
 It is easy to see that there exists $C>0$ such that $|\rho_k(t)|\le C$ for $k\ge k_0$ and $t\in [0,T]$. 
 
 We claim that 
 \begin{equation}\label{**}
 \lim_{k\to\infty}\rho_k(t)=0\quad(t\in [0,T]).
 \end{equation}
 From this claim and the Dominated Convergence Theorem, we then obtain
 \begin{equation}\label{*}
 \lim_{k\to\infty}(B_k-C_k)=0.
 \end{equation}
 Next we prove the claim in \eqref{**}. We fix an arbitrary $t_0\in(0,T)$. By Lemma~\ref{case4} it is enough to consider the following three cases:
\begin{description}
\item case~1: $x_i(t_0)\in\Omega$. There exists $k_1\ge k_0$ such that for any $k\ge k_1$ there holds $x_i^k(t_0)\in\Omega$. For any such $k$, we obtain $\chi_k(t_0)=0$ and $\chi_k^*(t_0)=0$, and thus $\rho_k(t_0)=0$.
  \smallskip
\item case~2: $x_i(t_0)\in\partial\Omega$ and $g(t_0)=0$. Since $\lim_{k\to\infty}g_k(t_0)=g(t_0)=0$, we obtain $\lim_{k\to\infty}\rho_k(t_0)=0$.
  \smallskip
\item case~3: $x_i(t_0)\in\partial\Omega$ and $g(t_0)>0$. There exists $k_1\ge k_0$ such that $g_k(t_0)>0$ for all $k\ge k_1$. Let such a $k$ be fixed.
If $x_i^k(t_0)\notin\overline{\Omega}$, then since $\chi_k(t_0)=\chi_k^*(t_0)=\chi_{\partial\Omega}(t_0)=1$, we obtain $\rho_k(t_0)=0$. Otherwise, if $x_i^k(t_0)\in\overline{\Omega}$, then $\chi_k(t_0)=\chi_k^*(t_0)=0$ and thus $\rho_k(t_0)=0$.
  \smallskip
\end{description}
Together, cases 1-3 imply \eqref{**}. 

In the final step we pass to the limit in $C_k$ as $k \to \infty$. Since $\chi_k^*(t)\in\{0,1\}$, there exists a subsequence of $\{\chi_k^*\}_k$~(not relabelled) and $\overline{\chi}\in L^{\infty}(0,T;[0,1])$ such that $\chi_k^*\rightharpoonup \overline{\chi}$ weakly-$*$ in $L^{\infty}(0,T)$ as $k\to\infty$. Then, from \eqref{g} and \eqref{h} we obtain
\[
\lim_{k\to\infty}C_k=\int_0^Tg(t)h(t)\overline{\chi}(t)\chi_{\partial\Omega}(t)dt.
\]
From this and \eqref{*}, Lemma \ref{lemmaB} follows.
\end{proof}

Using \eqref{lemmaA} and Lemma \ref{lemmaB}, we pass to the limit $k \to \infty$ in \eqref{ss}. This yields
\begin{equation}\label{eqweak2}
\int_0^T\dot{x_i}(t)\cdot\psi(t)dt=\int_0^TF_i(t)\cdot \psi(t)dt-\int_0^Tg(t)h(t)\overline{\chi}(t)\chi_{\partial\Omega}(t)dt.
\end{equation}
Since $\psi\in C_0^{\infty}((0,T);\R^m)$ is arbitrary, this implies that
\begin{equation}\label{a}
\dot{x}_i(t)=F_i(t)-g(t)\overline{\chi}(t)\chi_{\partial\Omega}(t)\nu(x_i(t)) \quad \text{for a.e.}~t\in(0,T).
\end{equation}

In the final step we show that \eqref{a} is equivalent to the second of the three conditions in Proposition \ref{prop:23}. 
We put $A:=\{t\in [0,T]~;~x_i(t)\in\partial\Omega\}$. Then, for a.e.~$t\in A$, we obtain
\[
\begin{aligned}
\frac{d}{dt}d_s^{\varepsilon}(x_i(t))&=\nabla d_s^{\varepsilon}(x_i(t))\cdot\dot{x}_i(t)\\
& =\nu(x_i(t))\cdot\{F_i(t)-g(t)\overline{\chi}(t)\chi_{\partial\Omega}(t)\nu(x_i(t))\}\\
&=g(t)(1-\overline{\chi}(t)).
\end{aligned}
\]
Applying the co-area formula~(Theorem~\ref{CA}) to $A$ and $d_s^{\varepsilon}\circ x_i(t)$, we obtain
\begin{equation}\label{b}
\begin{aligned}
\int_0^T|g(t)(1-\overline{\chi}(t))|\chi_{\partial\Omega}(t)dt&=\int_A |g(t)(1-\overline{\chi}(t))|dt\\
&=\int_A\big|\frac{d}{dt}(d_s^{\varepsilon}\circ x_i)(t)\big|dt\\
&=\int_{\R}\mathcal{H}^0(A\cap (d_s^{\varepsilon}\circ x_i)^{-1}(y))dy\\
&=0.
\end{aligned}
\end{equation}
In the last equality we have used that the integrand is $0$ a.e.~(in fact, the integrand is $0$ for any $y\neq 0$).
From \eqref{b}, we obtain
\begin{equation}\label{c}
g(t)\overline{\chi}(t)\chi_{\partial\Omega}(t)=g(t)\chi_{\partial\Omega}(t)\quad\text{for a.e.~}t\in [0,T].
\end{equation}
From \eqref{a}, using \eqref{c} and Lemma~\ref{case4}, we obtain
\[
\dot{x}_i(t)=F_i(t)-\big(\nu(x_i(t))\cdot F_i(t)\big)_+\nu(x_i(t))\chi_{\partial\Omega}(t) = H_i(t, X(t)) \quad\text{a.e.~}t\in [0,T].
\]
We conclude by applying Proposition \ref{prop:23}.

 \subsection{Uniqueness of the solution}
\label{Uniqueness}
 In this section, we prove the uniqueness statement of Theorem \ref{theorem1}. We suppose that $F_i$ satisfies the Lipschitz condition~\eqref{2}. Let $T > 0$ and $X^\circ \in G$ be given, and let $X(t)=\{x_i(t)\}_{i=1}^n$ and $Y(t)=\{y_i(t)\}_{i=1}^n$ be two mild solutions of Problem \ref{problem1} subject to the initial condition $X^\circ$. 
 
We prove uniqueness by applying Gronwall's lemma. With this aim (recalling that $X, Y$ are Lipschitz continuous by Proposition \ref{prop:23}) we set
\[ J_i(t):=\frac{d}{dt}\Big(\frac{1}{2}|x_i(t)-y_i(t)|^2\Big) \]
for $i\in\{1,2,\dots,n\}$ and a.e.~$t\in (0,T)$. The following lemma provides a sufficient bound on $J_i(t)$.

\begin{Lem}\label{lemmaU}
  There exists a constant $C > 0$ independent of $X$, $Y$ 
  such that
  \[
  J_i(t) \le C\|X(t)-Y(t)\|^2 \quad(i\in\{1,2,\dots,n\},~\text{a.e.~}t\in[0,T]).
  \]
\end{Lem}

\begin{proof}
For a.e. $t\in (0,T)$, we obtain
\[
\begin{aligned}
  J_i(t) &= \frac{d}{dt}\Big(\frac{1}{2}|x_i(t)-y_i(t)|^2\Big)\\
  &=(x_i(t)-y_i(t))\cdot(\dot{x_i}(t)-\dot{y_i}(t))\\
  &=(x_i(t)-y_i(t))\cdot \big( H_i(t,X(t))-H_i(t,Y(t)) \big).
  \end{aligned}
\]
We prove the required estimate on $J_i(t)$ by separating three cases. Since the computations below do not depend on time or on $i$, we simplify notation by setting $x=x_i(t)$, $y=y_i(t)$, $X=X(t)$, $Y=Y(t)$, $f(X)=F_i(t,X(t))$ and $J=J_i$.
\begin{description}
\item Case~1: $x\in\partial\Omega,~f(X)\cdot\nu(x)>0$ and $y\in\partial\Omega,~f(Y)\cdot \nu(y)>0$. Since $\nu$ and $f$ are Lipschitz continuous, we obtain
\[
\begin{aligned}
  J&= (x-y)\cdot\Big(\big(f(X)-(f(X)\cdot\nu(x))\nu(x)\big)-\big(f(Y)-(f(Y)\cdot\nu(y))\nu(y)\big)\Big)\\
   &\le C \|X-Y\|^2
  \end{aligned}
\]

\item Case~2: Both $x$ and $y$ do not satisfy the condition in Case~1. We obtain
\[
J=(x-y)\cdot(f(X)-f(Y))\le L\|X-Y\|^2.
\]

\item Case~3: $y\in\partial\Omega$, $f(Y)\cdot\nu(y)>0$ and $x$ does not satisfy the condition in Case~1. We compute
\[
J=(x-y)\cdot\big(f(X)- [ f(Y)-(f(Y)\cdot\nu(y))\nu(y) ] \big).
\]

To complete the estimate, we separate a further set of three cases.
\item Case~3.1: $|x-y|\ge\varepsilon$. We obtain
\[
J\le |x-y|(2\|F\|_{\infty})\le C_{\varepsilon}|x-y|^2\le C_{\varepsilon}\|X-Y\|^2,
\]
where $C_{\varepsilon}=\frac{2\|F\|_{\infty}}{\varepsilon}$.

\item Case~3.2: $|x-y| < \varepsilon$ and $(x-y)\cdot\nu(y)\le 0$ (see Figure \ref{examplef5} for a sketch). We obtain
\[
\begin{aligned}
  J&=(x-y)\cdot(f(X)-f(Y)) + \big( f(Y)\cdot\nu(y) \big) \big( (x-y)\cdot\nu(y) \big) \\
  &\le (x-y)\cdot (f(X)-f(Y))\\
  &\le L\|X-Y\|^2 .
  \end{aligned}
\]
\begin{figure}[htbp]
\centering
\begin{minipage}[t]{0.48\textwidth}
\centering
\includegraphics[width=85mm]{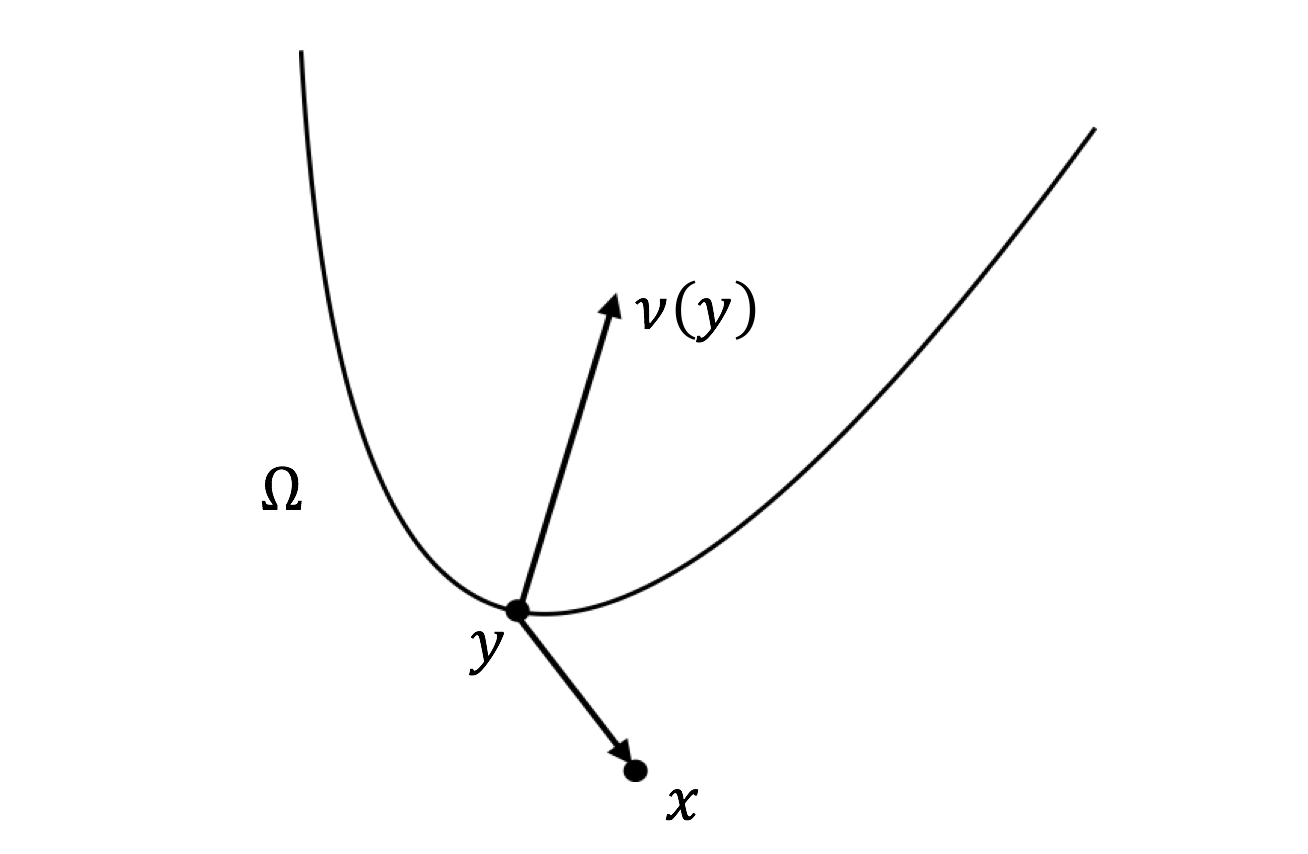}
\caption{Typical scenario of Case~3.2.}
\label{examplef5}
\end{minipage}
\begin{minipage}[t]{0.48\textwidth}
\centering
\includegraphics[width=85mm]{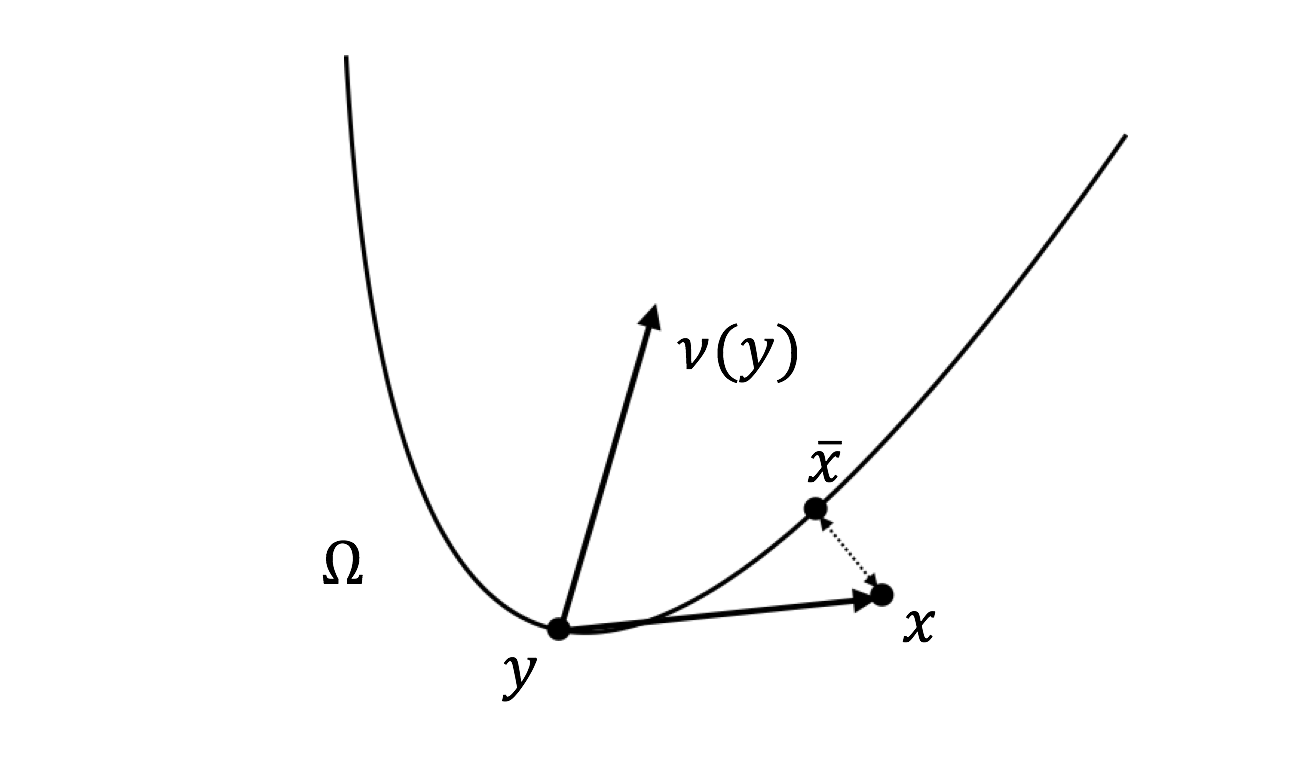}
\caption{Typical scenario of Case~3.3.}
\label{examplef6}
\end{minipage}
\end{figure}
    
\item Case~3.3: $|x-y| < \varepsilon$ and $(x-y)\cdot\nu(y)>0$ (see Figure \ref{examplef6} for a sketch). We estimate
\[
\begin{aligned}
  J&=(x-y)\cdot(f(X)-f(Y))
  + \big( f(Y)\cdot\nu(y) \big) \big( (x-y)\cdot\nu(y) \big) \\
  &\leq L \| X-Y \|^2 + \big( f(Y)\cdot\nu(y) \big) \big( (x-y)\cdot\nu(y) \big).
  \end{aligned}
\]
Since $0 < f(Y)\cdot\nu(y) < \| F \|_\infty$, it is enough to show that
\begin{equation} \label{pf:1}
  (x-y)\cdot\nu(y) \leq C |x-y|^2
\end{equation}
for some $C > 0$ which is independent of $x$ and $y$. With this aim, we set
\[
\bar{x} := \argmin_{z\in\partial\Omega}|x - z|,
\]
where $\bar{x}$ is uniquely determined since $x\in N^{\varepsilon}(\partial\Omega)\cap \overline{\Omega}$. We notet that
\begin{equation}\label{ineqA}
|\bar{x}-y|\le |\bar{x}-x|+|x-y|\le |y-x|+|x-y|=2|x-y|.
\end{equation}
We expand the left-hand side of \eqref{pf:1} as
\begin{equation} \label{pf:2}
  (x-y)\cdot\nu(y) =(x-\bar{x})\cdot\nu(\bar{x})
      +(\bar{x}-y)\cdot\nu(\bar{x})
      +(x-y)\cdot (\nu(y)-\nu(\bar{x})).
\end{equation}

Next we show that each of the three terms in the right hand side of \eqref{pf:2} is bounded from above by $C\|X-Y\|^2$. For the first term, we find $(x-\bar{x})\cdot\nu(\bar{x})=-|x-\bar{x}|\le0$. For the third term, we estimate
\[
(x-y)\cdot(\nu(y)-\nu(\bar{x}))
\le |x-y|C|y-\bar{x}|
\le 2C|x-y|^2
\le 2C \|X-Y\|^2.
\]
For the second term in \eqref{pf:2}, we obtain by using a Taylor expansion of $d_s$ that 
\[
0=d_s(y)
 =d_s(\bar{x}+(y-\bar{x}))
 =d_s(\bar{x})+(y-\bar{x})\cdot\nabla d_s(\bar{x})+R(\bar{x},y),
\]
where $|R(\bar{x},y)|\le C|\bar{x}-y|^2$. Hence,
\[
(\bar{x}-y)\cdot\nu(\bar{x})=(\bar{x}-y)\cdot\nabla d_s(\bar{x})=R(y,\bar{x})\le C|\bar{x}-y|^2\le 4C|x-y|^2 \le 4C\|X-Y\|^2.
\]
This completes the proof of \eqref{pf:1}.
\end{description}
\end{proof}

To complete the proof of the uniqueness statement of Theorem \ref{theorem1}, we continue from Lemma \ref{lemmaU} by using Gronwall's Lemma. Integrating the result from Lemma~\ref{lemmaU} from $0$ to $t$, we obtain for any $1\le i\le n$ that
\[
\begin{aligned}
  \frac{1}{2}|x_i(t)-y_i(t)|^2 
  \le \frac{1}{2}|x_i(0)-y_i(0)|^2+C\int_0^t\|X(s)-Y(s)\|^2ds
  = C\int_0^t\|X(s)-Y(s)\|^2ds.
  \end{aligned}
\]
Since the constant $C$ does not depend on $i$, we can take the maximum over $i$ on both sides to obtain
\[
\|X(t)-Y(t)\|^2 \le 2C\int_0^t\|X(s)-Y(s)\|^2ds\quad(t\in [0,T]).
\]
Hence, by Gronwall's Lemma, 
\[
X(t)=Y(t) \quad \text{for all } t \in [0,T].
\]
This completes the proof of the uniqueness statement of Theorem \ref{theorem1}.

\subsection{Properties of the solutions}
\label{Property of the solution}

We prove two properties of solutions to Problem \ref{problem1}. As we will demonstrate by simulations in Section~\ref{sec4}, particles can attach to or detach from the boundary. The first property states that particles can detach only in tangential direction. Again, we will simply write $F_i(t) = F_i(t, X(t))$.

\begin{Prop}\label{property}
Suppose that \eqref{1} holds and $X(t)~(0\le t\le T)$ is a solution of Problem~\ref{problem1}. If $x_i(t_0)\in\partial\Omega$ at $t_0\in (0,T]$, then it holds that
\begin{equation}\label{fnu}
F_i(t_0)\cdot \nu(x_i(t_0))\ge 0.
\end{equation}
Moreover, if there exists a $\delta$ such that $x_i(t)\in\Omega$ for all $t_0<t<t_0+\delta\le T$, then it holds that
\[
\frac{d^+}{dt}x_i(t_0) := \lim_{\tau\to +0}\frac{x_i(t_0+\tau)-x_i(t_0)}{\tau}=F_i(t_0)
\quad \text{and} \quad
F_i(t_0)\cdot\nu(x_i(t_0))=0.
\]
\end{Prop}
\begin{proof}
As in \eqref{h}, we set $g(t):=\big(\nabla d_s^{\varepsilon}(x_i(t)))\cdot F_i(t)\big)$. We prove $g(t_0)=F_i(t_0)\cdot \nu(x_i(t_0))\ge 0$ by contradiction. Suppose that $g(t_0)<0$. Then there exists an $h>0$ such that $g(t)<0$ for all $t\in[t_0-h, t_0+h]$. From the definition of $H_i$, we then obtain that $\dot{x}_i(t) = F_i(t)$ for a.e.~$t\in[t_0-h,t_0+h]$. Hence, $x_i$ is differentiable at $t_0$, and we obtain
\[
\frac{d (d_s^{\varepsilon} \circ x_i)}{dt} (t_0)
= \nabla d_s^{\varepsilon}(x_i(t_0))\cdot \dot{x}_i(t_0)
= g(t_0) < 0,
\]
which contradicts with $x_i(t) \in \overline \Omega$. Therefore, $F_i(t_0)\cdot\nu(x_i(t_0))\ge 0$.

Next we prove the second statement of Proposition \ref{property}. Let the asserted $\delta$ be given. As before, we obtain that $x_i \in C^1([t_0, t_0+\delta];\R^m)$, and thus
\[
\begin{aligned}
\frac{d^+}{dt}x_i(t_0) =\lim_{\tau\to +0}\frac{1}{\tau}\int_{t_0}^{t_0+\tau}\dot{x}_i(t)dt
 =\lim_{\tau\to +0}\frac{1}{\tau}\int_{t_0}^{t_0+\tau}F_i(t)dt=F_i(t_0)
\end{aligned}
\]
and 
\[
0\ge \frac{d^+}{dt}d_s^{\varepsilon}(x_i(t_0))=\nabla d_s^{\varepsilon}(x_i(t_0))\cdot \frac{d^+}{dt}x_i(t_0)=\nu(x_i(t_0))\cdot F_i(t_0).
\]
Combining this with \eqref{fnu}, we also obtain $F_i(t_0)\cdot\nu(x_i(t_0))=0$.
\end{proof}

The second property specifies the sense in which, under the Lipschitz condition \eqref{2} on $F$, how the solution to Problem \ref{problem2} approximates the solution to Problem \ref{problem1}.

\begin{Th}[Approximation theorem]
If $F_i$ satisfies \eqref{1} and the Lipschitz condition \eqref{2}, the solution $X^k(t)$ of Problem~\ref{problem2} is unique for each $k>\frac{1}{\varepsilon}\|F\|_{\infty}$ and it uniformly converges to the unique solution $X(t)$ of Problem~\ref{problem1} on $[0,T]$ as $k\to\infty$.
\end{Th}

\begin{proof}
The uniqueness of $X^k(t)$ follows from the Lipschitz continuity of $F_i$ and $d\nabla d$~(Proposition~\ref{d}). By the argument in the proof of the existence of $X$, any uniformly converging subsequence of $\{X^k\}_k$ converges to a solution $Y$ of Problem~\ref{problem1}. Since the solution of Problem~\ref{problem1} is unique, we conclude that the whole sequence $X^k$ converges uniformly to $X$ on $[0,T]$ as $k\to\infty$.
\end{proof}

\section{Gradient flows of particle interaction energies}\label{sec3}
\setcounter{equation}{0}
\subsection{Particle interaction energy and its minimization problem}

Let $\Omega$ and $X = \{x_i\}_{i=1}^n \in G$ be as in Section~\ref{ss:PnT}. To each state $X \in G$ we assign an energy value which is the sum of an interaction potential $V(x_i-x_j)$ (summed over each pair of two particles $x_i$ and $x_j$) and a potential energy $W(x_i)$ (summed over $i$). We assume that the potentials $V$ and $W$ satisfy
\[
\begin{aligned}
 (V1) &:~V \in C^2(\R^m\setminus\{0\};\R),~\lim\limits_{x\to 0}V(x)=\infty \text{ and } V(-x)=V(x)>0\quad(x\in\R^m\setminus \{0\}),\\
 (W1) &:~W\in C^2(\overline{\Omega}) \text{ and } W(x)\ge 0.
\end{aligned}
\]
Figure \ref{V} illustrates a typical example of $V$.
\begin{figure}[htbp]
      \begin{center}
        \includegraphics[width=90mm]{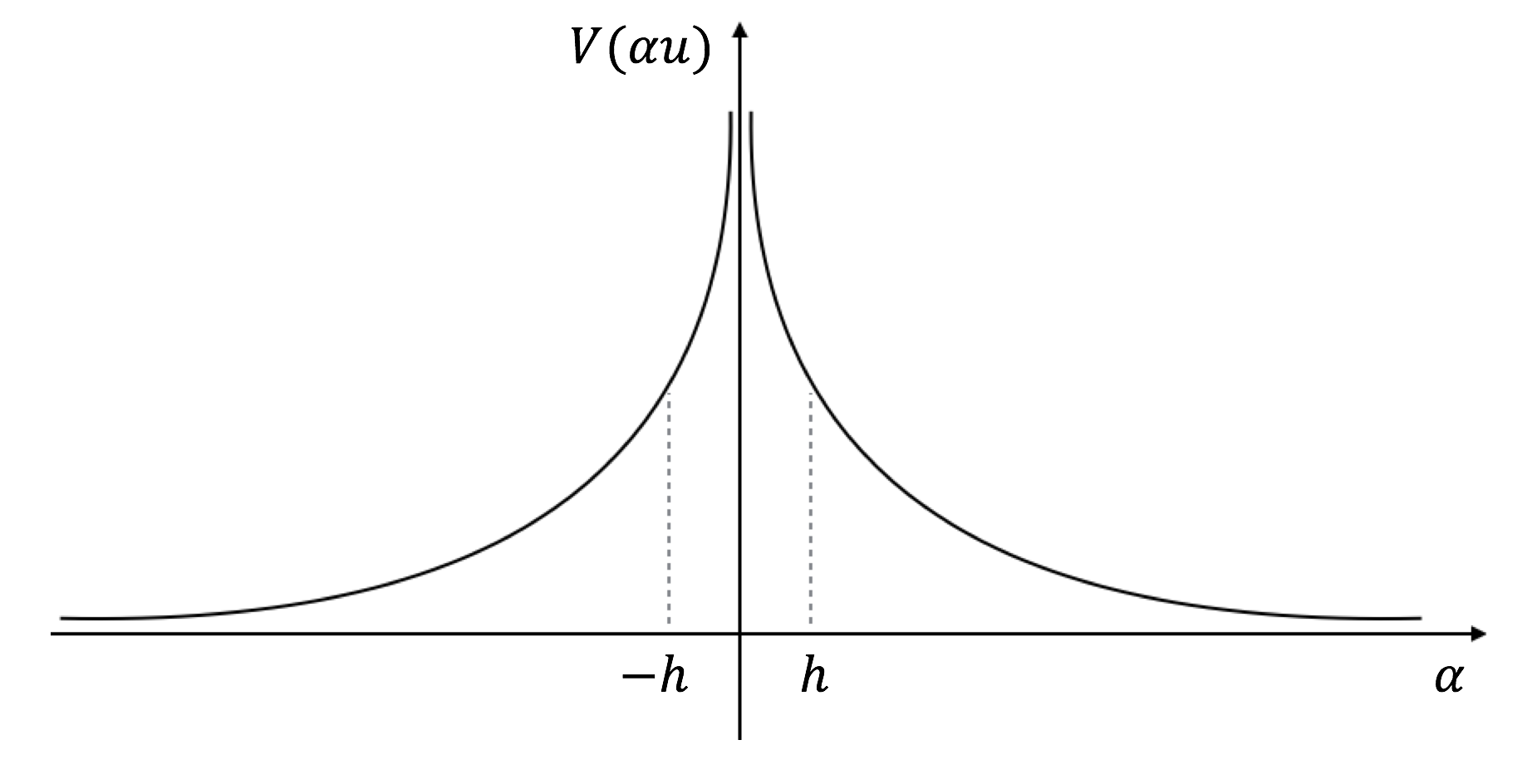}
      \end{center}
      \caption{Typical shape of the interaction potential $V$ along a line $\alpha \mapsto \alpha u$ with $u \in \R^m$ and $|u| = 1$.}
      \label{V}
\end{figure}

We set
\[
G_*:=\{X\in G~;~x_i\neq x_j~\text{for all } i\neq j\}
\]
as the set of states where the interaction energy is finite. The, we define the particle energy $E:G_*\to\R$ as
\begin{equation}\label{E}
  E(X):=\frac{1}{n(n-1)}\sum_{i=1}^n\sum_{j=1}^{i-1}V(x_{i}-x_{j})+\frac{1}{n}\sum_{i=1}^{n}W(x_{i}).
  \end{equation}
We consider the following minimization problem.
\begin{Prob}\label{problemm}
  Find $\overline{X} \in G_*$ such that
  \[
  \overline{X}=\argmin_{X\in G_*} E(X).
  \]
\end{Prob}

Existence of solutions to Problem \ref{problemm} follows from the lower semi-continuity of $E$ and $G$ being closed and bounded. In preparation for constructing a gradient flow of $E$, we prove that the particles within a particle configuration $X$ with finite energy have a minimal length of separation.

\begin{Lem}\label{lemmaV}
  For all $X\in G_*$, there exists an $h=h(E(X))>0$ non-increasing with respect to the value of $E(X)$, such that
  \begin{align}\label{hp1}
  &\min_{i\neq j}|x_i-x_j|
  > h, \text{ and} \\ \label{hp2}
  &V(y) 
  > (n-1)nE(X) \quad \text{for all~} |y| \leq h.
  \end{align}
\end{Lem}

\begin{proof}
Since $V, W \geq 0$, we obtain
\begin{equation} \label{p:hp1}
  V(x_i-x_j) 
  \le \sum_{1\le k<l \le n}V(x_k-x_l) 
  \le (n-1)nE(X) \quad \text{for any } i\neq j.
\end{equation}
Since $V(x) \to \infty$ as $|x| \to 0$, there exists an $h>0$ such that for all $|x| \leq h$, $V(x) > (n-1)nE(X)$. Together with \eqref{hp1}, this also proves \eqref{hp1}.
\end{proof}

\subsection{Gradient flow of the interaction energy}
We consider $n$ moving particles $X(t)\in G_*$. For a particle $x_i=(x_i^1,\dots,x_i^m)^{\top}$, we set
\[
\nabla_i:=
\left(
\begin{array}{c}
  \frac{\partial}{\partial x_i^1}\\
  \vdots\\
  \frac{\partial}{\partial x_i^m}\\
\end{array}
\right)
\]
and
\begin{equation}\label{eqF}
F_i(X):=-\nabla_iE(X).
\end{equation}

In this setting, the problem corresponding to Problem \eqref{problem1} is:
\begin{Prob}\label{problem3}
Let $P$ be as in \eqref{P}, $F_i$ as in \eqref{eqF}, $X^{\circ}=\{x_i^{\circ}\}_{i=1}^n\in G_*$ and $T>0$. Find a solution to
\begin{equation} \label{P33:eqn}
\left\{
\begin{aligned}
  & \dot{x_i}(t)=P\big(x_i(t),F_i(X(t))\big)\quad\big(\text{for a.e.}~t\in [0,T]\big)\\
  & x_i(0)=x_i^{\circ},
\end{aligned}
  \right.
\end{equation}
  for $i=1,2,\dots,n$.
 \end{Prob}
 
Similar to Problem \ref{problem1}, we do not expect solutions to Problem~\ref{problem3} to be of class $C^1$. Therefore, we seek again mild solutions (see Definition \ref{sol}) to Problem \ref{problem3}. Our main result in this section (Theorem \ref{Th3}) is the existence and uniqueness of the solution to Problem~\ref{problem3}.
\begin{Th}\label{Th3}
 Let $E$ be as in \eqref{E}. Then for all $T>0$ and all $X^{\circ}\in G_*$, there exists a unique mild solution (see Definition~\ref{sol}) to Problem~\ref{problem3}.
\end{Th}

\subsection{Proof of Theorem~\ref{Th3}: existence and uniqueness}

In this section we prove Theorem~\ref{Th3}. Let $T > 0$ and $X^\circ \in G^*$ be given. Since $V$ is singular at 0, $F_i$ does not satisfy the Lipschitz condition in Section~\ref{ss:PnT}. Hence, we cannot apply directly Theorem~\ref{theorem1} to prove Theorem~\ref{Th3}. 

We construct a solution to Problem \ref{problem3} by introducing an auxiliary problem which fits to the assumptions in Section \ref{ss:PnT} and whose solution will turn out to be a solution to Problem \ref{problem3} too. With this aim, we first prove an energy-decay estimate under the assumption that Problem \ref{problem3} attains a solution.

\begin{Prop}\label{prop}
  Let $X$ be a solution to Problem~\ref{problem3}. Then for a.e.~$t\in (0,T)$, $\frac{d}{dt}E(X(t))\le 0$.
\end{Prop}

\begin{proof}
Setting
\[
\chi_i(t):=\left\{
\begin{aligned}
  & 1\qquad\text{if}~x_i(t)\in\partial\Omega,~F_i(X(t))\cdot\nu(x_i(t))>0,\\
  & 0\qquad\text{otherwise},
\end{aligned}
  \right.
\]
we rewrite \eqref{P33:eqn} as
  \[
  \dot{x}_i(t)=-\nabla_i E(X(t)) + \chi_i(t)\big(\nabla_i E(X(t))\cdot\nu(x_i(t))\big)\nu(x_i(t))
  \]
  for a.e.~$t\in [0,T]$, $i=1,2,\dots,n$. Hence, we obtain
  \[
  \begin{aligned}
    \frac{d}{dt}E(X(t))
    &=\sum_{i=1}^n\nabla_iE(X(t))\cdot\dot{x_i}(t)\\
    &=\sum_{i=1}^n\big(-|\nabla_i E(X(t))|^2 + \chi_i(t)|\nabla_i E(X(t))\cdot\nu(x_i(t))|^2\big)\\
    &\le 0
  \end{aligned}
  \]
  for a.e.~$t\in [0,T]$.
 \end{proof}

Using Lemma \ref{lemmaV}, we obtain an $h = h(E(X^{\circ})) > 0$ such that
\[
  \min_{i\neq j}|x_i^\circ - x_j^\circ| > h.
\]
Using $h$, we choose a regular potential $\widetilde{V}\in C^2(\R^m;\R)$ which satisfies
\begin{equation}\label{potential}
\widetilde{V}(x)\left\{
\begin{aligned}
& =V(x)  &(|x|\ge h)\\
&\ge\min_{|y|\le h}V(y) &(|x|< h)
\end{aligned}
\right.
\end{equation}
Figure \ref{wideV} illustrates a typical example of $\widetilde{V}$. 
\begin{figure}[htbp] 
      \begin{center}
        \includegraphics[width=90mm]{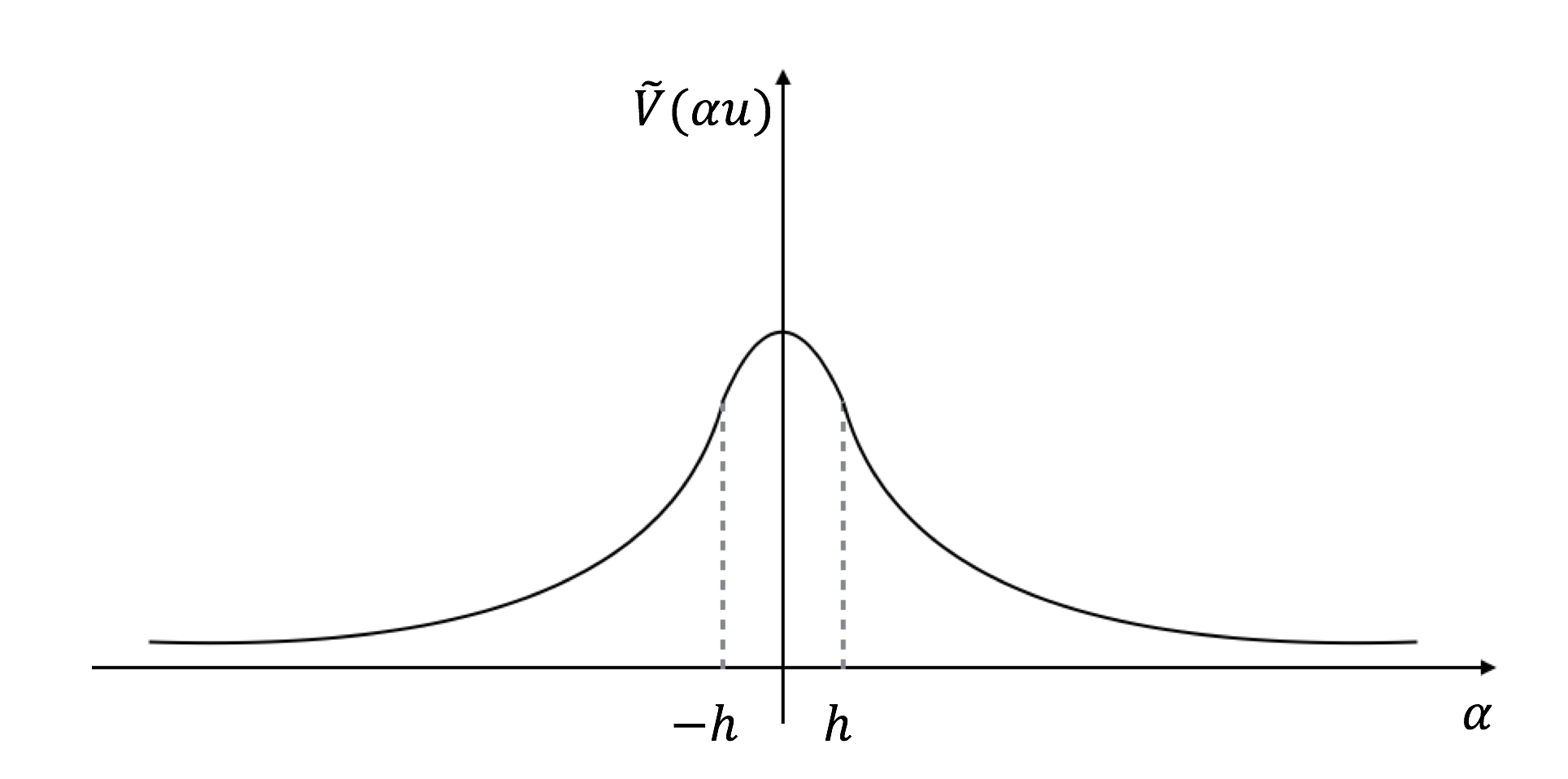}
      \end{center}
      \caption{A typical graph of $\widetilde{V}$ as seen along the same line through the origin as used in Figure \ref{V}.}
      \label{wideV}
\end{figure}

Using $\widetilde V$, we define the regularized particle energy $\widetilde{E}:G\to \R$ by
\begin{equation}\label{WE}
\widetilde{E}(X):=\frac{1}{n(n-1)}\sum_{1\le i<j\le n}\widetilde{V}(x_i-x_j)+\frac{1}{n}\sum_{i=1}^nW(x_i).
\end{equation}
By \eqref{potential}, we note that $\widetilde{E} = E$ on
\begin{equation}\label{Gh}
G_h:=\{X\in G~;~\min_{i\neq j}|x_i-x_j|\ge h\}.
\end{equation}

Using the preparations above, we finally introduce the auxiliary problem.
\begin{Prob}\label{problem4}
Let $P$ be as in \eqref{P}, $\widetilde{F}_i := -\nabla_i\widetilde{E}$, $\widetilde X^{\circ}\in G$ and $\widetilde T>0$. Find a solution to
\[
\left\{
\begin{aligned}
  & \dot{x_i} = P(x_i(t),\widetilde{F}_i(X(t)))\quad\mathrm{a.e.}~t \in (0, \widetilde T)\\
  & x_i(0)= \widetilde x_i^{\circ}\in\Omega
\end{aligned}
  \right.
  \]
  for $i=1,2,\dots,n$.
 \end{Prob}

If follows directly from Theorem \ref{theorem1} that Problem \ref{problem4} attains a unique solution for any $\widetilde X^{\circ}\in G$ and $\widetilde T>0$. Choosing $\widetilde X^{\circ} = X^{\circ}$ and $\widetilde T = T$, we denote the corresponding solution to Problem \ref{problem4} by $\widetilde X$. It is obvious from the proof of Proposition \ref{prop} that $\widetilde E$ and $\widetilde X$ also satisfy $\frac{d}{dt} \widetilde E( \widetilde X(t))\le 0$ for all $t \in (0,T)$. Following a similar argument as in the proof of Lemma \ref{lemmaV}, we then obtain
\[ 
\max_{i \neq j} \widetilde V( \widetilde x_i(t) - \widetilde x_j(t) ) 
< (n-1) n \widetilde E (\widetilde{X}(t))
\le (n-1)n \widetilde E( \widetilde X^{\circ})
= (n-1)n E(X^{\circ})
 \] 
for all $t\in [0,T]$. 
From \eqref{hp2} and \eqref{potential}, we obtain for any $|x|\le h$ that
\[
(n-1)n E(X^{\circ})<\min_{|y|\le h} V(y)\le \widetilde{V}(x).
\] 
From the two inequalities above we conclude that $\widetilde X (t) \in G_h$ for all $t\in [0,T]$. Therefore, since $\widetilde E = E$ on $G_h$, $\widetilde{X}$ is also a solution to Problem~\ref{problem3}.
\medskip

Next we prove the uniqueness of solutions to Problem \ref{problem3}. If $X$ is a solution to Problem~\ref{problem3} with initial condition $X^{\circ}$, we obtain from Proposition \ref{prop} that $E(X(t))$ is non-increasing in time. Hence,  Lemma~\ref{lemmaV} provides an $h=h(E(X^{\circ})) > 0$ such that $X (t) \in G_h$ for all $t\in [0,T]$.
Setting again $\widetilde{V}\in C^2(\R^m;\R)$ and $\widetilde E$ as above, $X$ satisfies Problem~\ref{problem4}. Since  the solution to Problem~\ref{problem4} is unique, Problem~\ref{problem3} cannot have any other solution.

\section{Numerical examples}\label{sec4}
\setcounter{equation}{0}
In this section we compute and discuss numerical solutions to the gradient flow with confinement as given by Problem \ref{problem3}. We solve Problem~\ref{problem3} numerically by discretizing in time with a fixed time step $\Delta t$. We compute the solution for the next time step by first applying the classical fourth order Runge-Kutta method without considering the confinement of the particles to the domain $\Omega$, and then projecting all particles outside $\Omega$ along the normal $\nu$ to $\partial \Omega$.

We solve Problem~\ref{problem3} for several different scenarios. In all scenarios, we choose the interaction potential as 
\[
V(x)=\frac{1}{|x|}.
\]
We split the scenarios in three sections for different choices of the domain $\Omega$.

\subsection{Circular domain $\Omega$} \label{sec41}

We set $\Omega := B_1(0) := \{|x|<1\}$ and consider no external force (i.e., $W \equiv 0$). We further take $n = 3000$ and $\Delta t = 0.5$ as reference values. Regarding the initial positions $X^\circ$ for $n$ particles, we consider two cases:
\begin{itemize}
  \item[Case 1] Similar to $G_*$ in Section \ref{sec3}, we set 
  \[
G_1 := \{X\in B_1(0)^n~;~ |x_i- x_j| \geq 0.025~\text{for all } i\neq j\},
\]
where we interpret $0.025$ as the separation distance. We choose $X^{\circ}\in G_1$ randomly uniformly. Figure~\ref{circase1} illustrates the realisation of $X^{\circ}$ that we use in our simulation.
  \item[Case 2] Similar to Case 1, we choose $X^\circ$ from the uniform distribution on 
  \[
G_2 := \{X\in B_{1/2}(1/2)^n~;~ |x_i-x_j| \geq 0.012~\text{for all } i\neq j\}.
\] 
Figure \ref{circase2} illustrates an example of $X^\circ$.
\end{itemize}

Figure \ref{circase1} illustrates the particle positions $X(t)$ at $t = 3000$. Firstly, we observe that the density of the particles is higher near the boundary. This is in line with the potential $V$ being decreasing in radial direction. Also, from Proposition \ref{property} we know that the particles cannot detach from the boundary after hitting it\footnote{To see this, suppose $x_i(t) \in \partial \Omega$. Then for all $j \neq i$, the force that particle $j$ exerts on particle $i$ has direction $x_i(t) - x_j(t)$. Together with the strict convexity of $\Omega$, we then conclude that $F_i(t, X(t)) \cdot \nu(x_i(t)) > 0$.}. Secondly, the color-coding suggests that no mixing of particles occurs. 

\begin{figure}[htbp]
\centering
\subfloat{\includegraphics[width=0.25\textwidth]{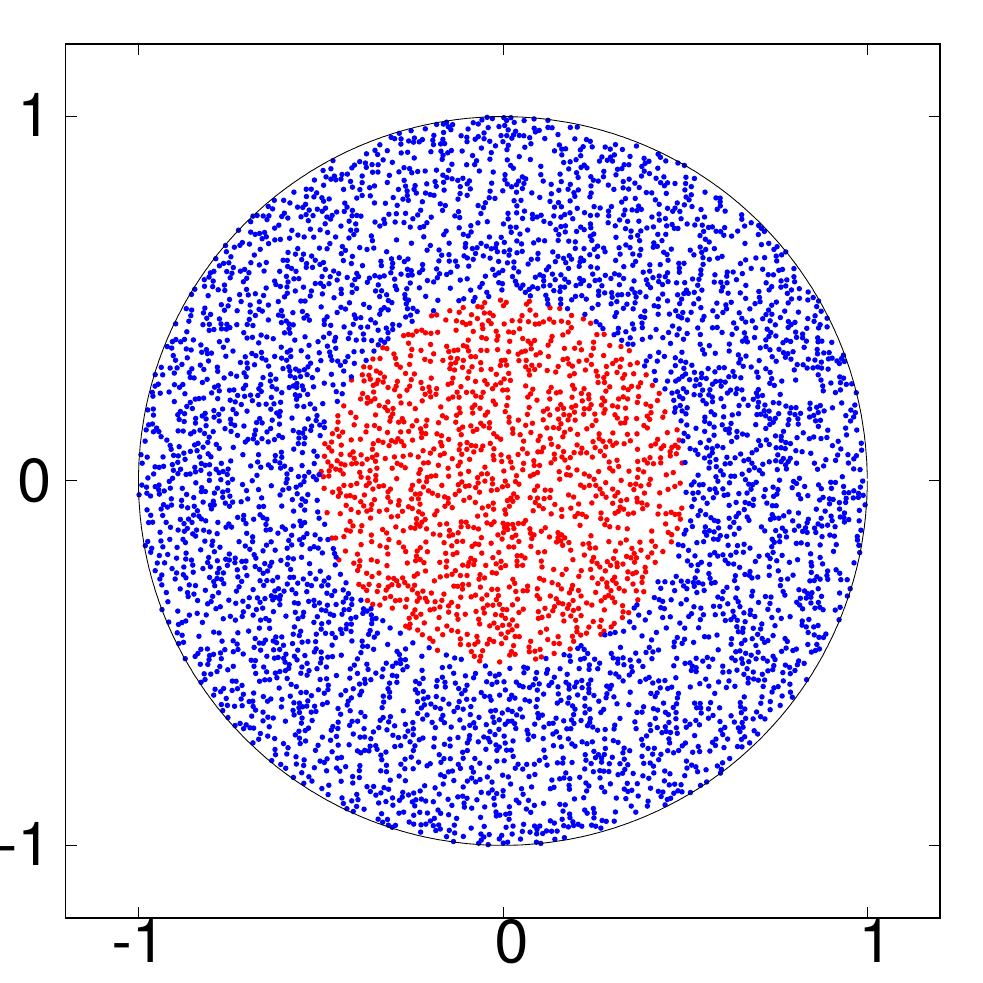}}
\subfloat{\includegraphics[width=0.25\textwidth]{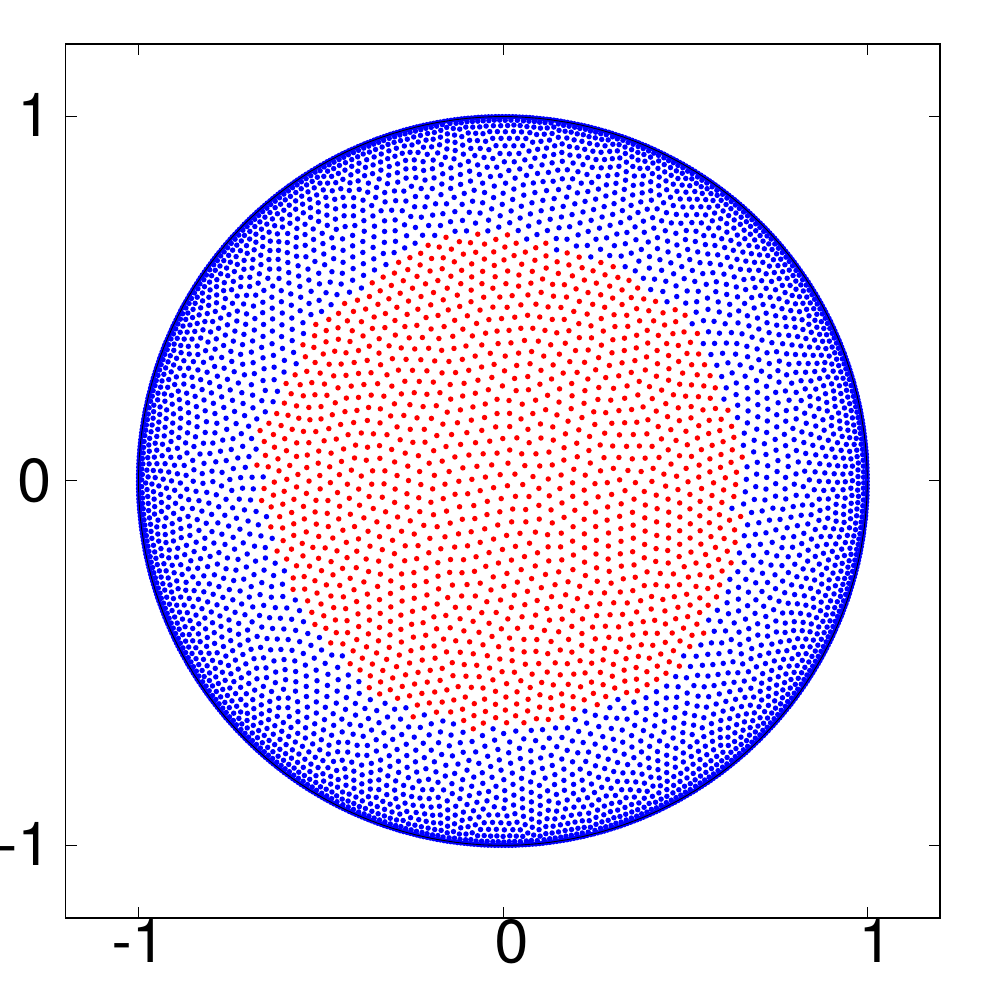}}
\caption{Choice of $X^\circ$ (left figure) as in Case~1 and the related numerical solution of Problem \ref{problem3} at $t = 3000$ (right figure). Solely for the sake of visualisation we use two colors for the particles: particles are red if $|x_i^\circ| < 0.5$ and blue otherwise.}
\label{circase1}
\end{figure}

Next we solve Problem \ref{problem3} for the initial condition as in Case 2. Figure~\ref{circase2} illustrates the numerical solution $X$ at several time instances. As in Case~1, we observe no mixing between the red and the blue particles. Also, the particle positions at $t = 3000$ are comparable. A new effect which we observe (especially in the second figure) is that several blue particles are separated from the bulk of the blue particles. A possible explanation for this effect is that we use an explicit time discretisation which is unstable if $\Delta t$ is not small enough. 

\begin{figure}[htbp]
\centering
\subfloat{\includegraphics[width=0.24\textwidth]{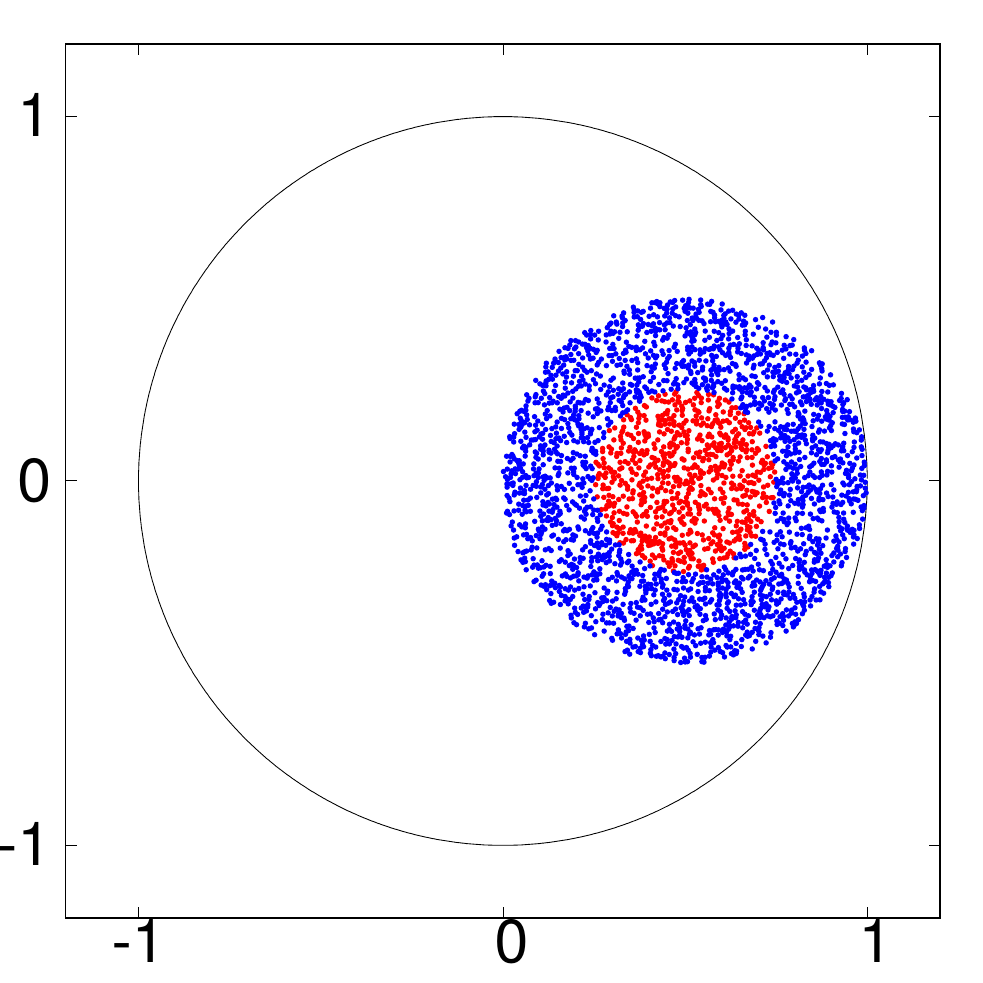}}
\subfloat{\includegraphics[width=0.24\textwidth]{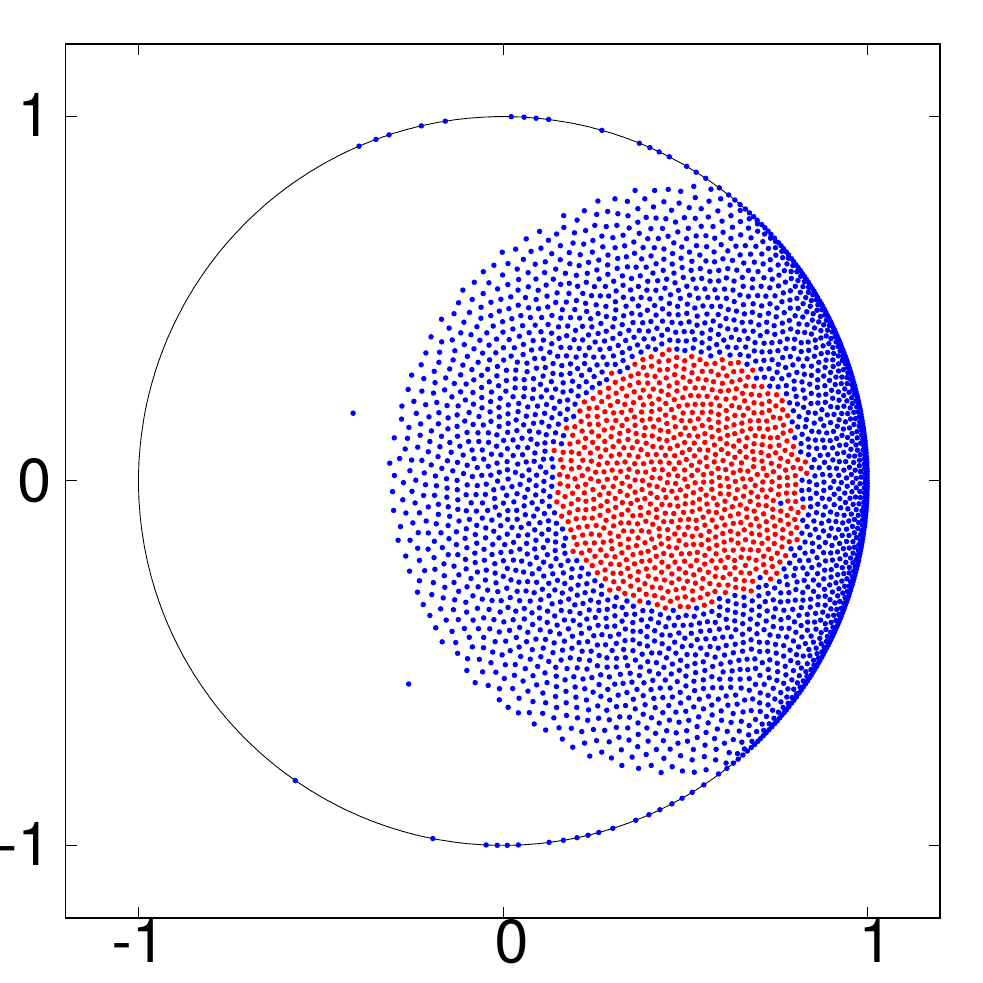}}
\subfloat{\includegraphics[width=0.24\textwidth]{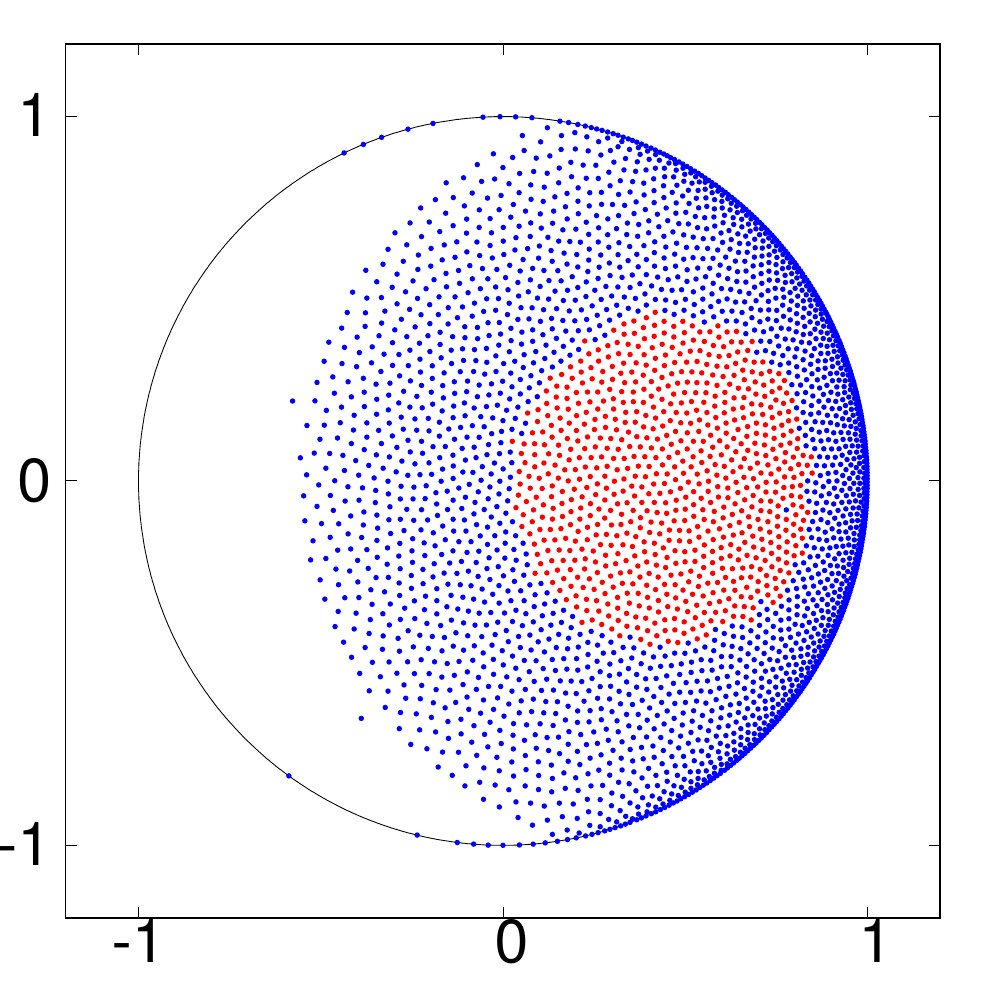}}
\subfloat{\includegraphics[width=0.24\textwidth]{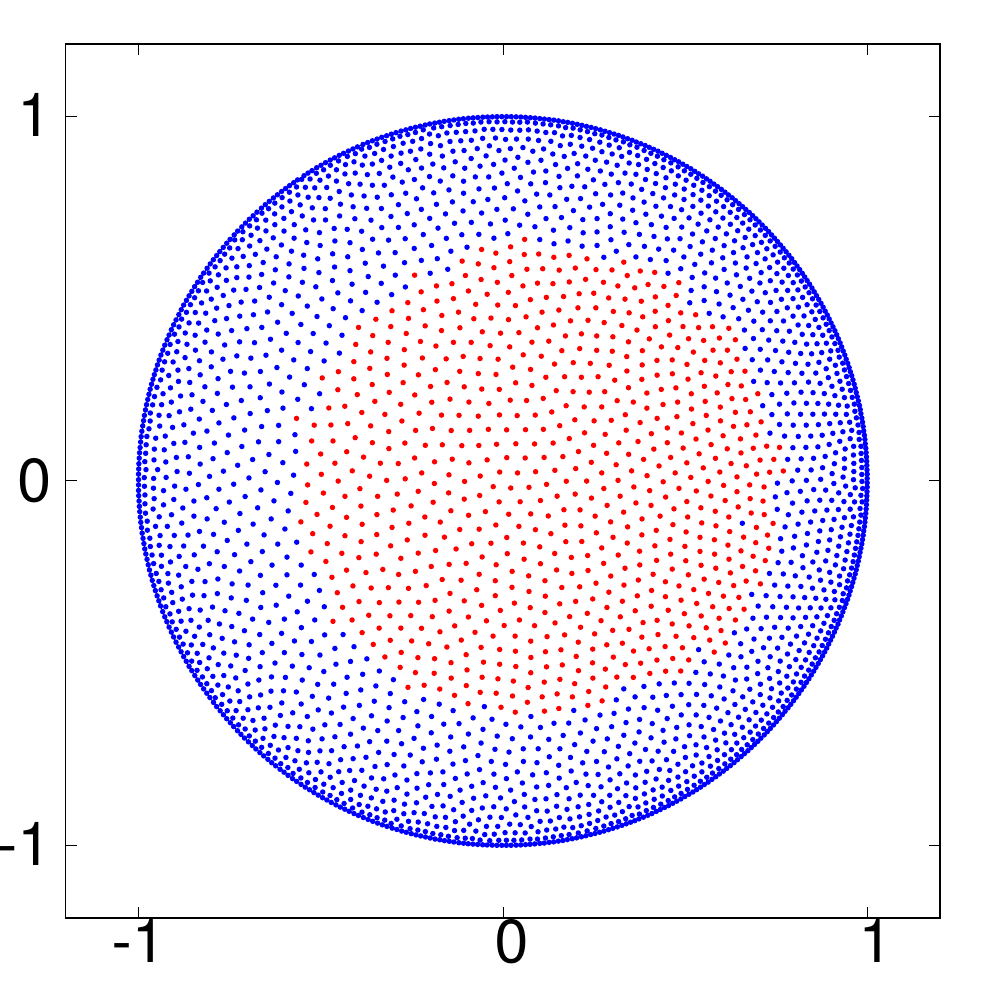}}
\caption{Choice of $X^\circ$ (left figure) as in Case~2 and the related numerical solution of Problem \ref{problem3} at $t=200$, $t=500$ and $t=3000$ respectively in the subsequent three figures. Particles satisfying $|x_i^\circ - (\frac12,0)^\top| < \frac14$ are displayed in red.}
\label{circase2}
\end{figure}

To gain insight in the dependence of the numerical solution, we consider the energy value $E(X(t))$ along the solutions in Case 1 and Case 2 for different values of $\Delta t$ and $n$. The results are displayed in Figure~\ref{Energycompare1}. We first discuss the case where $\Delta t = 3$. The graph displays peaks which violate the property of the analytical solution to Problem \ref{problem3} on the decay of the energy (see Proposition \ref{prop}). A possible explanation for these peaks is the (unstable) explicit time-discretization, which may cause the interaction energy to become large when particles come too close to each other. Interestingly, this instability occurs throughout the complete time interval.

Next we examine the figure where five different values of $n$ are considered. We observe for $t$ large enough that the energy increases as $n$ increases, but that the rate of the gap between these energy values decreases as $n$ increases. Not only does this observation suggest that solutions to Problem \ref{problem3} (after rescaling time by a factor $n$) may converge as $n \to \infty$ to a dynamical particle density, it also hints to a convergence \emph{rate}.

\begin{figure}[htbp]
\centering
\subfloat{\includegraphics[width=0.48\textwidth]{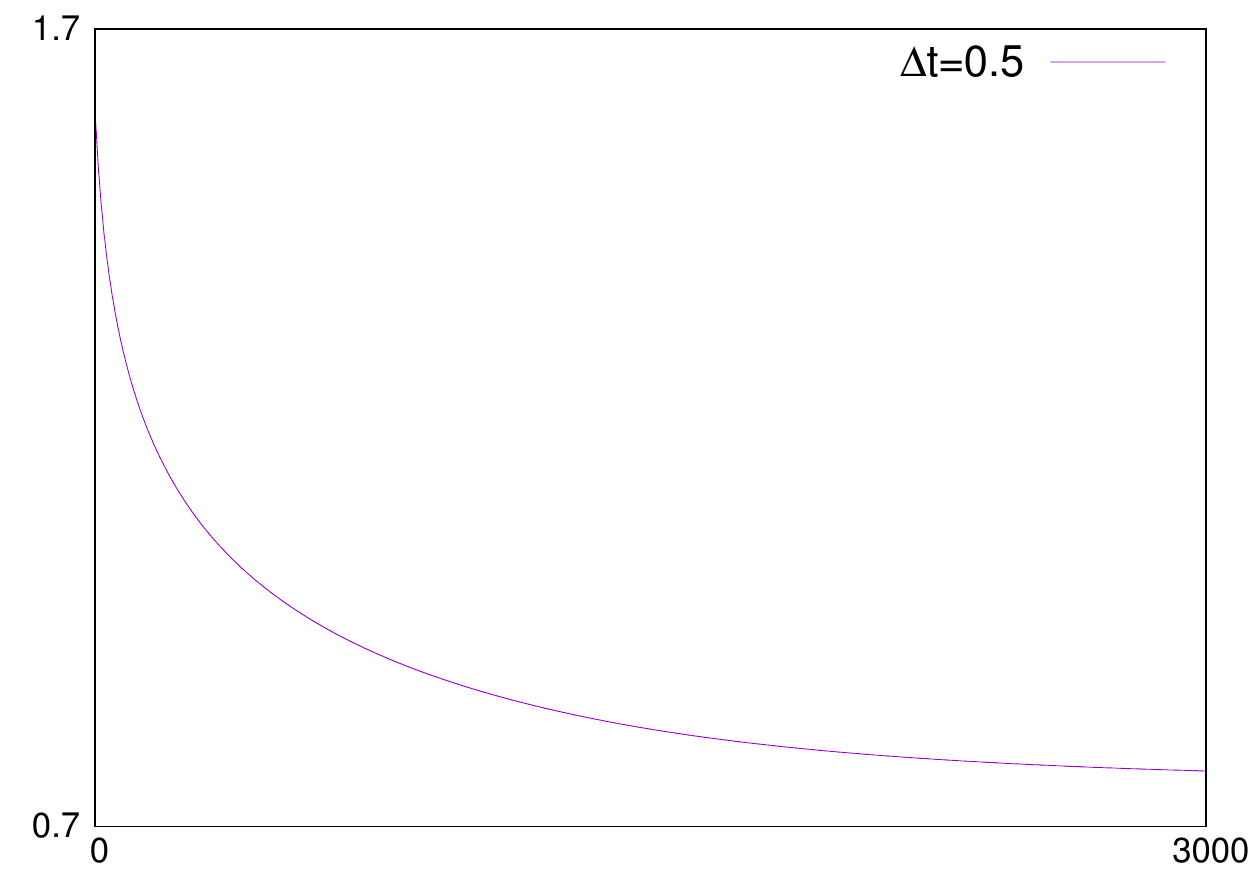}}
\subfloat{\includegraphics[width=0.48\textwidth]{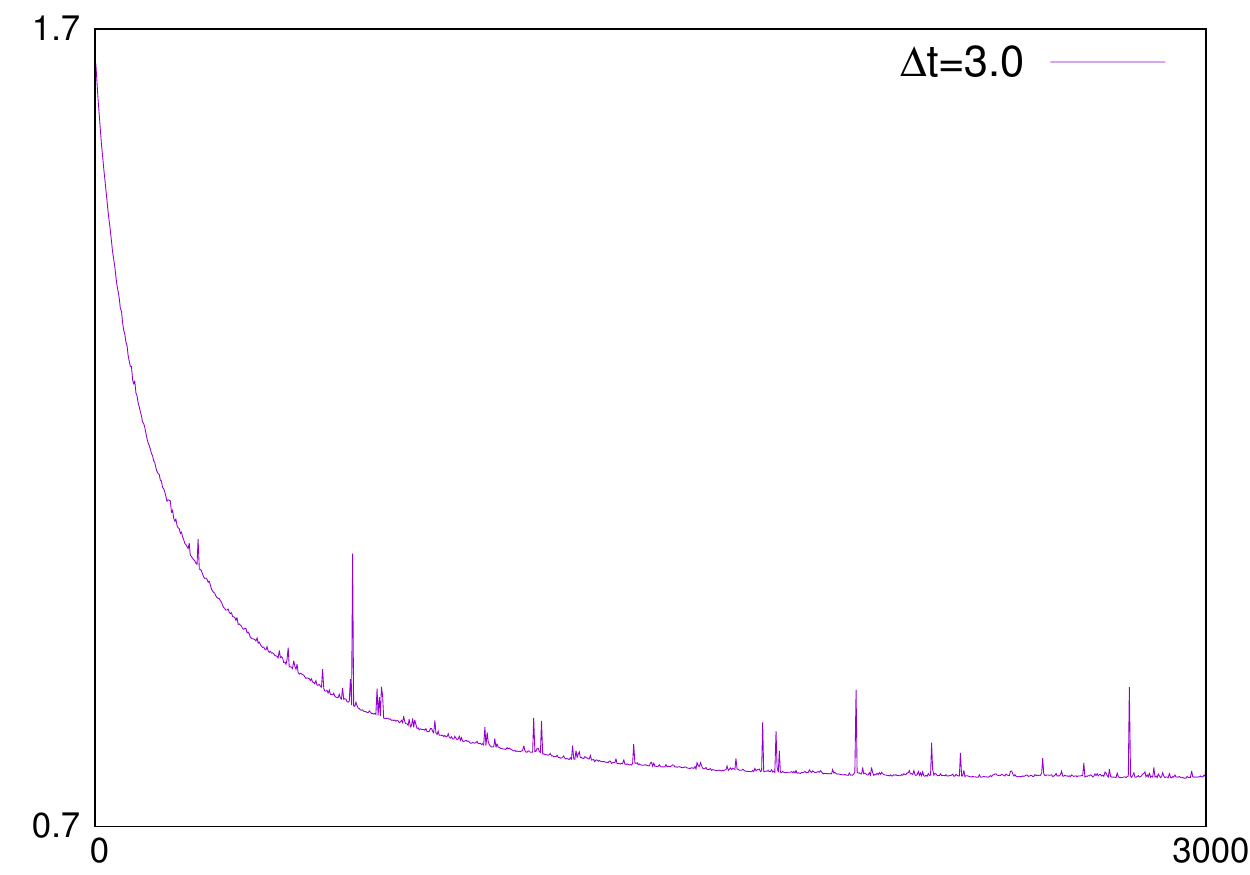}}\\
\subfloat{\includegraphics[width=0.48\textwidth]{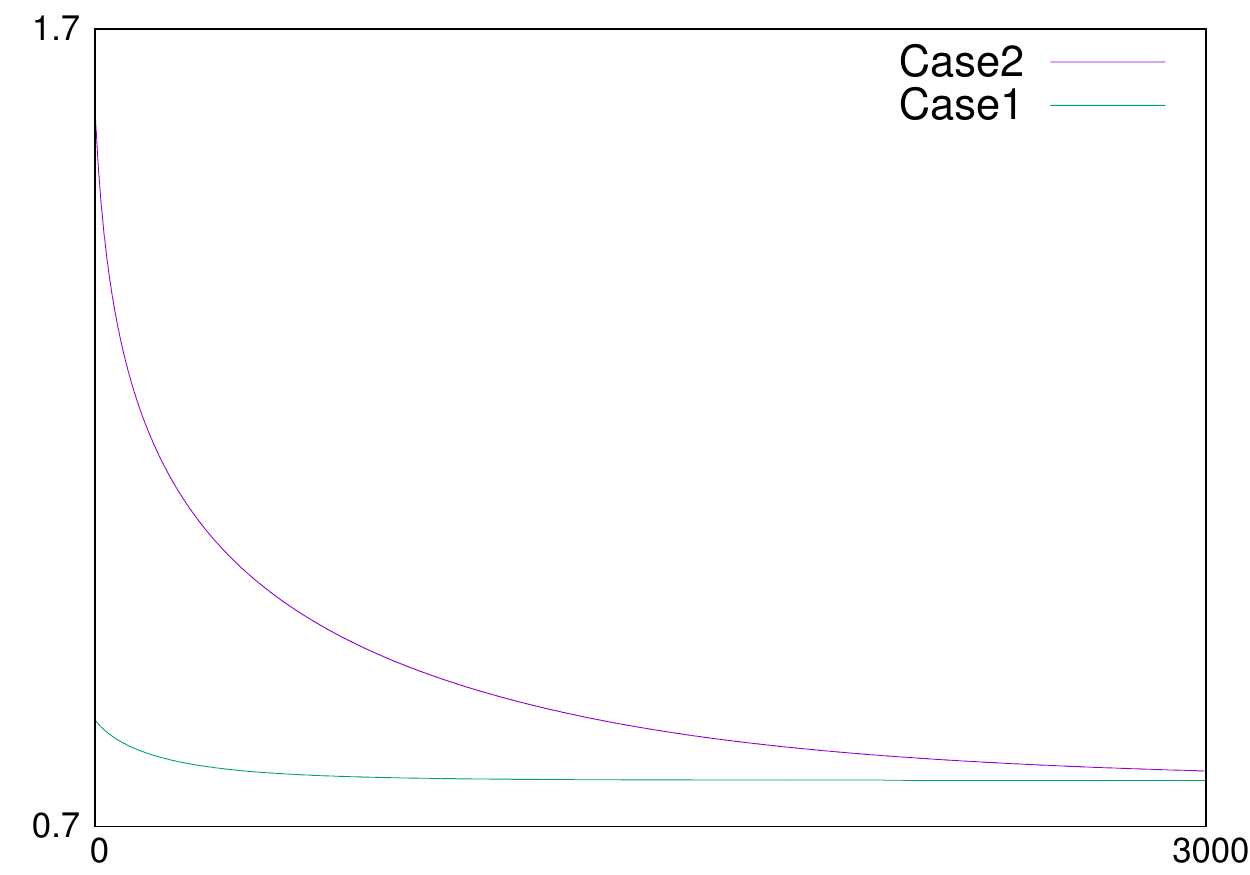}}
\subfloat{\includegraphics[width=0.48\textwidth]{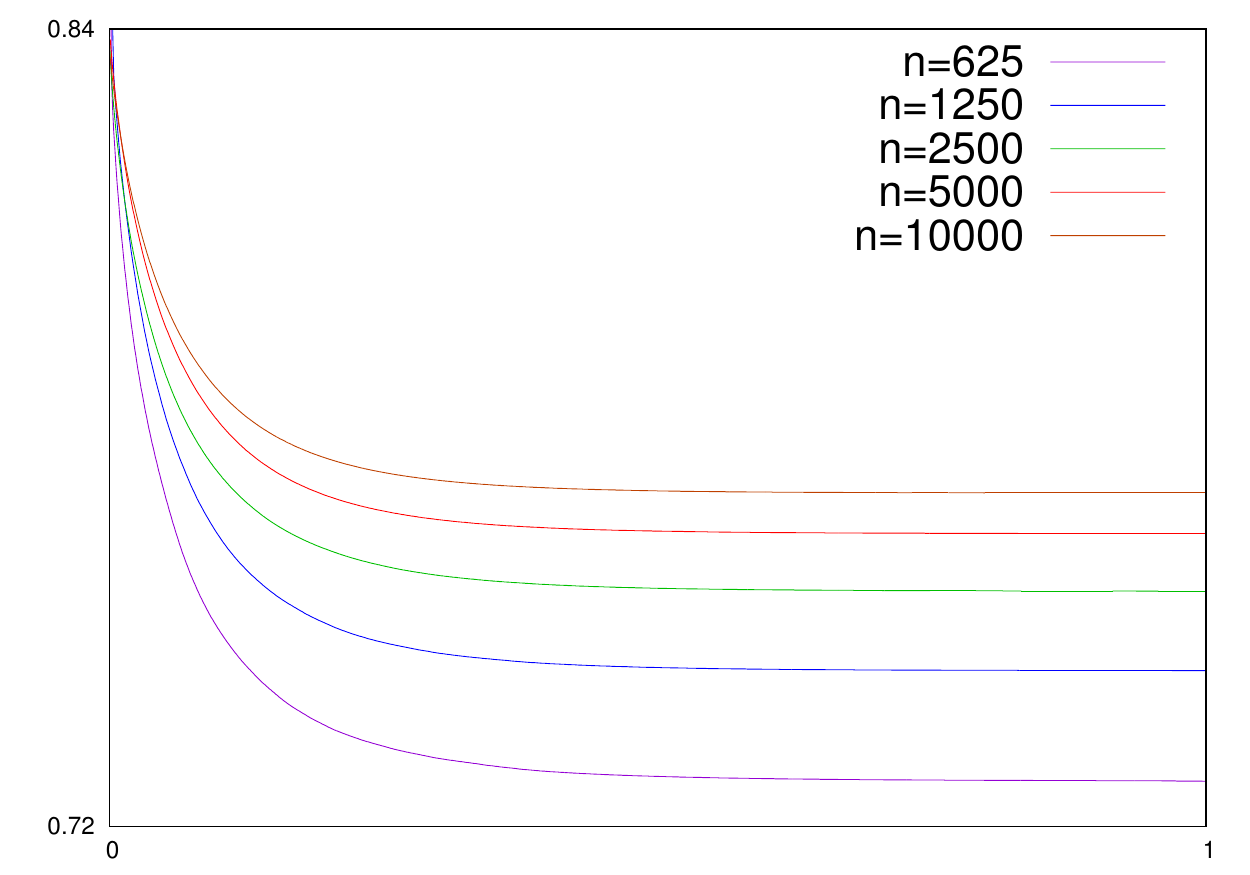}}
\caption{All four figures display the energy value (vertical axis) at the solution in time (horizontal axis). The top left figure corresponds to Case 2. In the bottom left figure we compare Case 2 with Case 1. The top right figure is obtained by recomputing Case 2 for $\Delta t = 3$ instead of $\Delta t = 0.5$. Finally, the bottom right figure compares the energy decay of five instances of Case~1 (with a rescaling of time by a factor $n$) for different values of $n$. }
\label{Energycompare1}
\end{figure}

\subsection{A non-convex domain $\Omega$ (dumbbell)} \label{sec42}
We choose $\Omega$ to be the non-convex domain as illustrated in Figure \ref{ncircase}. $\Omega$ is chosen such that the initial positions $x_i^\circ$ of Case 2 of Section \ref{sec41} fit inside $\Omega$ and such that $\Omega$ is included in the domain considered in Section \ref{sec41}. Hence, it is natural to compare this setting to Case 2 of Section \ref{sec41}. Other than $\Omega$, we will therefore make as few adjustments to the setting of Case 2 as possible. We only change $\Delta t$ to $\Delta t = 0.25$ to prevent numerical instabilities.

Figure \ref{ncircase} illustrates the solution to Problem \ref{problem3} at different time instances, and Figure \ref{ncirenergy} displays the corresponding decay of the energy. We observe that the particles, despite the non-convexity, spread out through the dumbbell in time such that the amount of particles on the left side is comparable to that on the right side. 
As in Section \ref{sec41} we observe no mixing between the red part and blue particles. Regarding the energy decay, we observe that our confinement of the domain with respect to Case 2 in Section \ref{sec41} yields a slower decrease of the energy. 

\begin{figure}[htbp]
\centering
\subfloat{\includegraphics[width=0.4\textwidth]{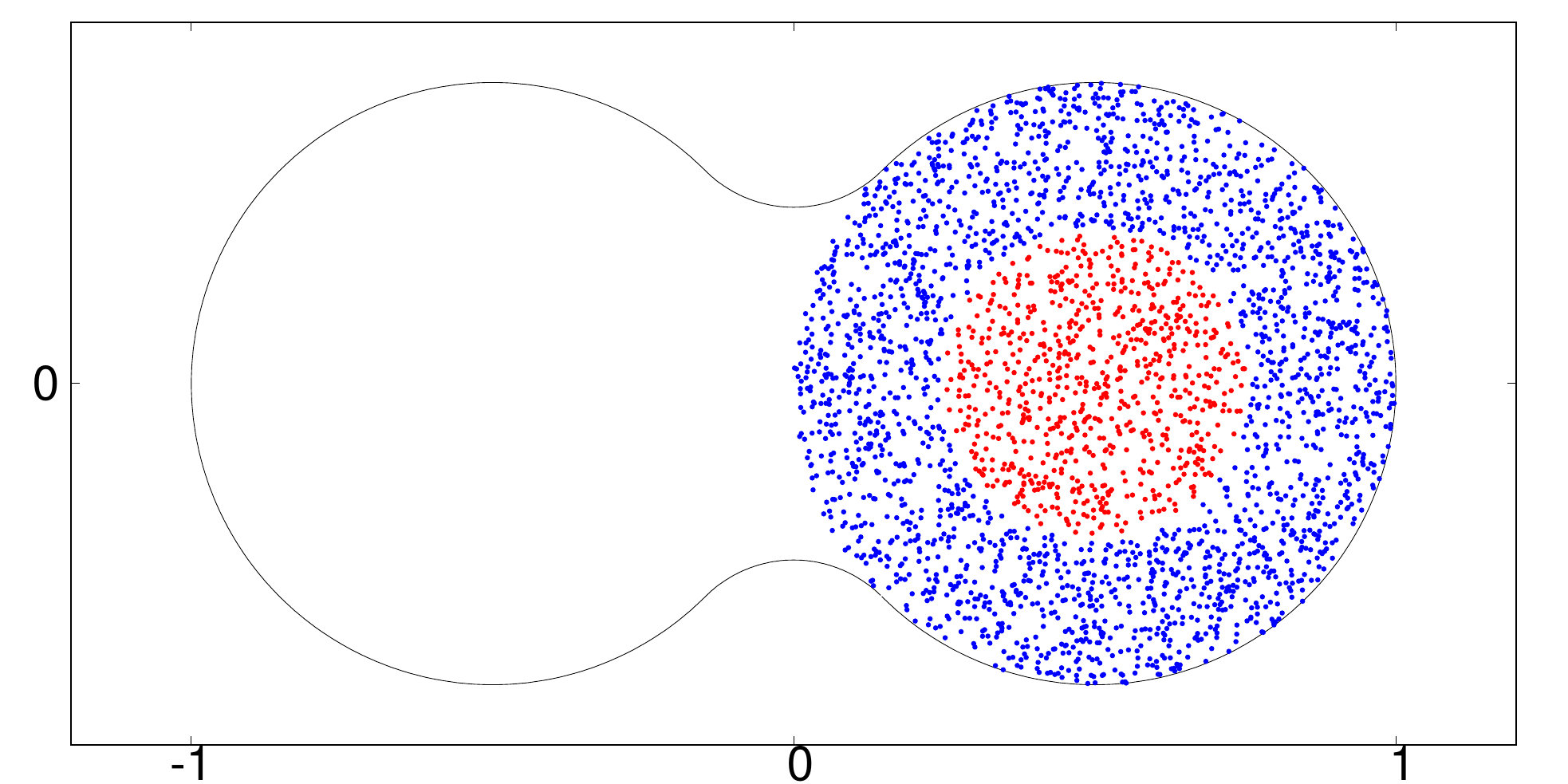}}
\subfloat{\includegraphics[width=0.4\textwidth]{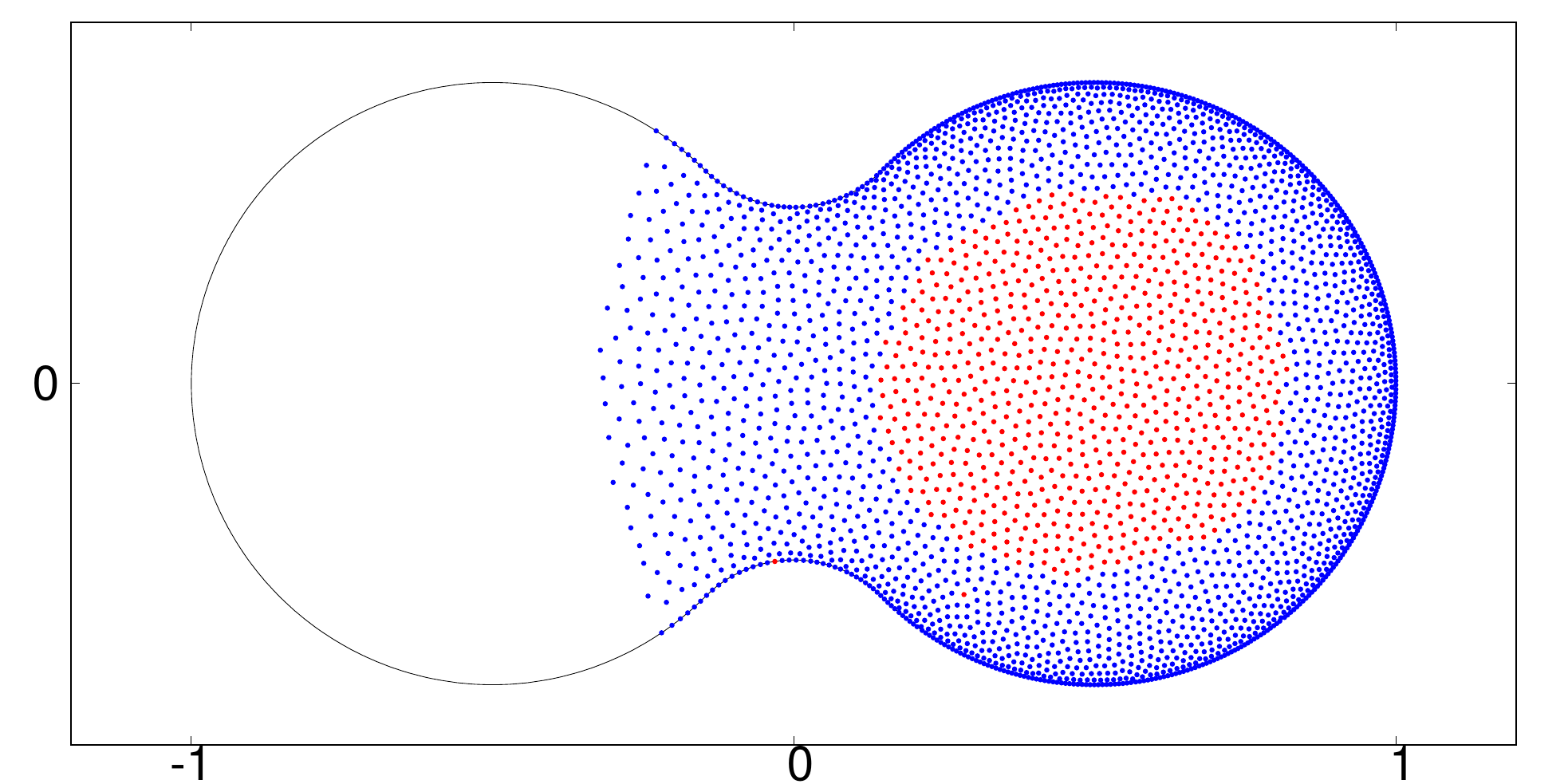}}\\
\subfloat{\includegraphics[width=0.4\textwidth]{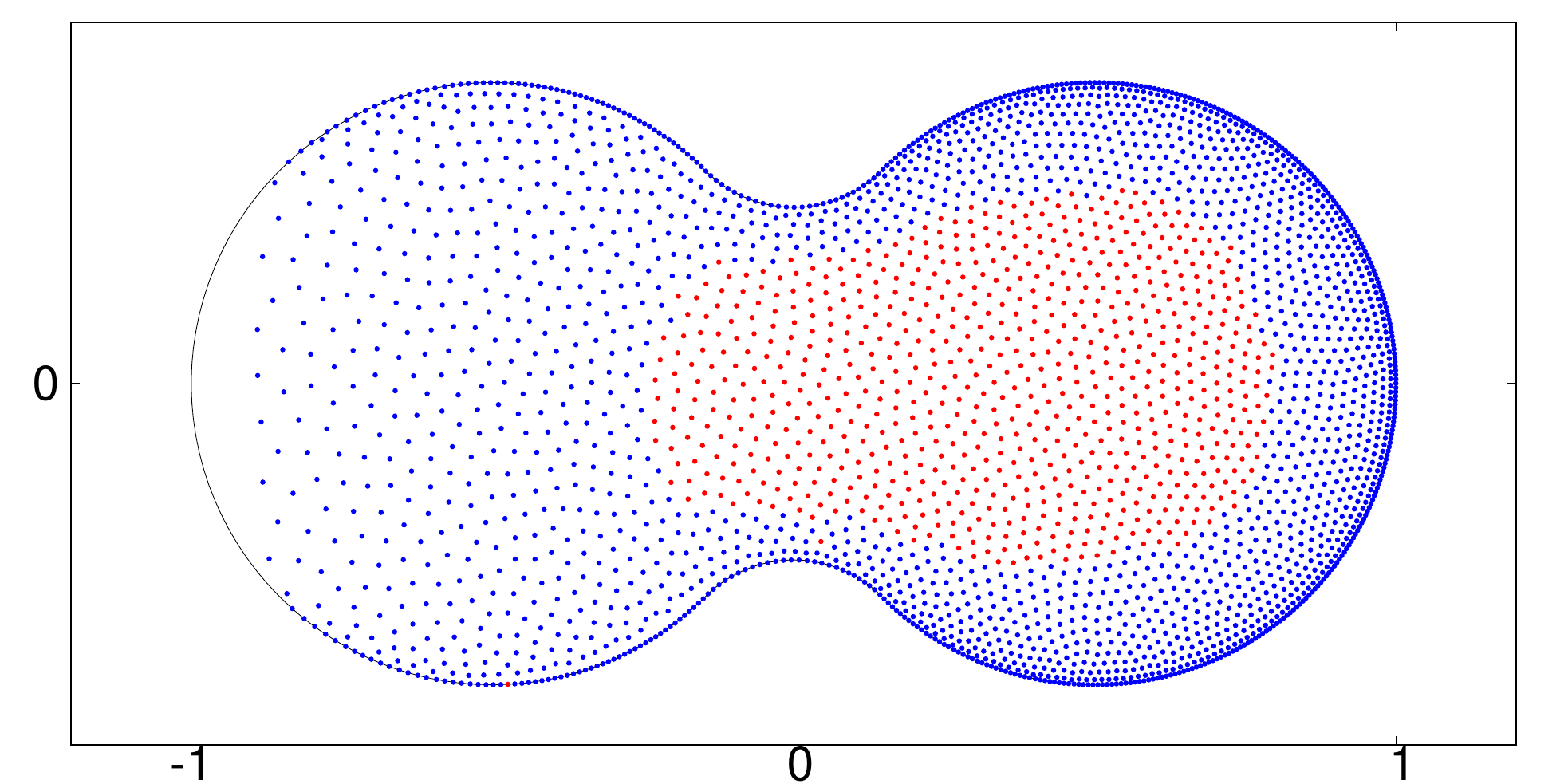}}
\subfloat{\includegraphics[width=0.4\textwidth]{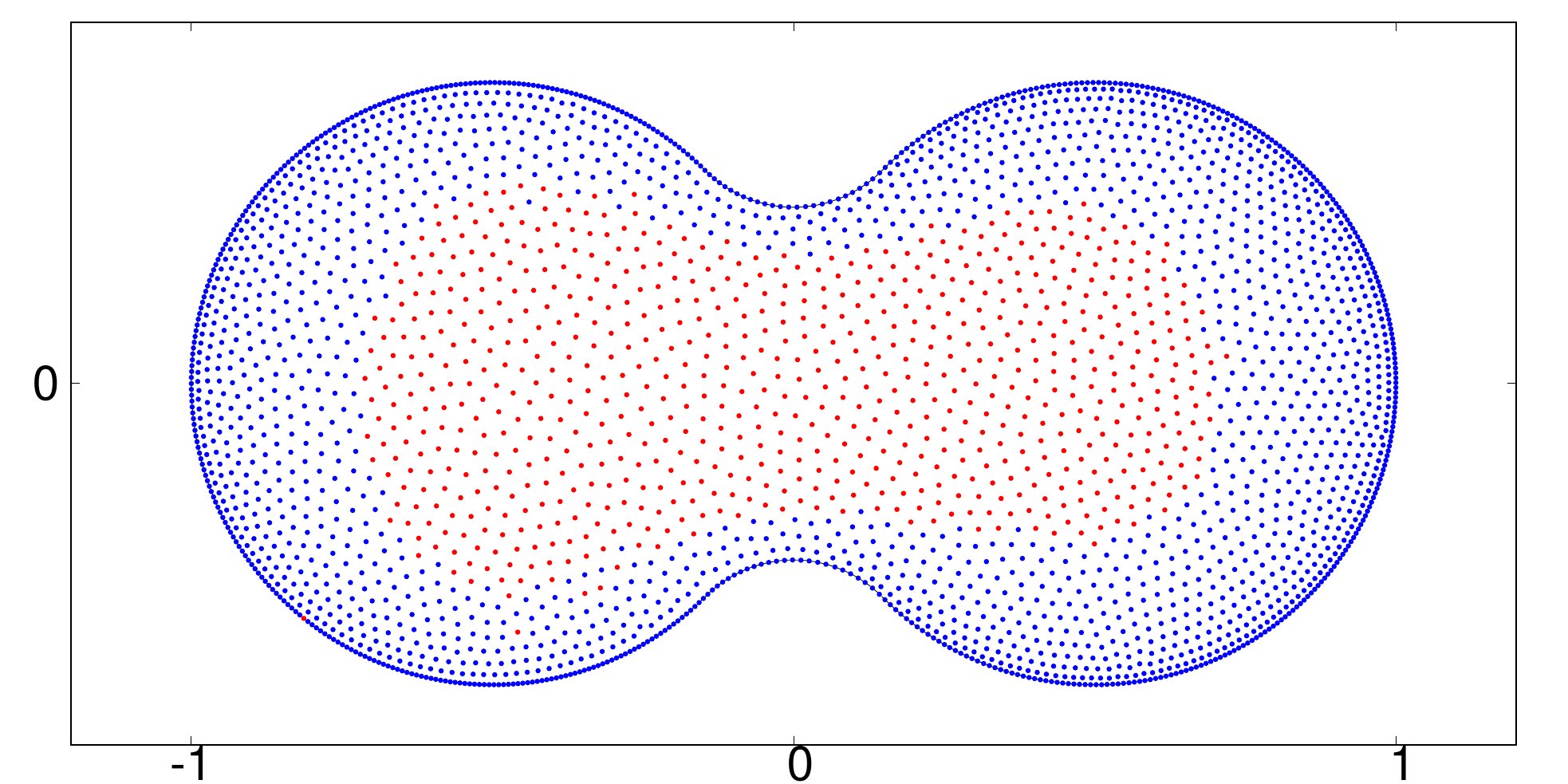}}
\caption{Aside from a different domain $\Omega$ and $\Delta t = 0.25$, the setting is the same as in Figure \ref{circase2}. The four plots correspond to the time instances $t=0$, $t=100$, $t=250$ and $t=2000$.}
\label{ncircase}
\end{figure}

\begin{figure}[htbp]
      \begin{center}
        \includegraphics[width=80mm]{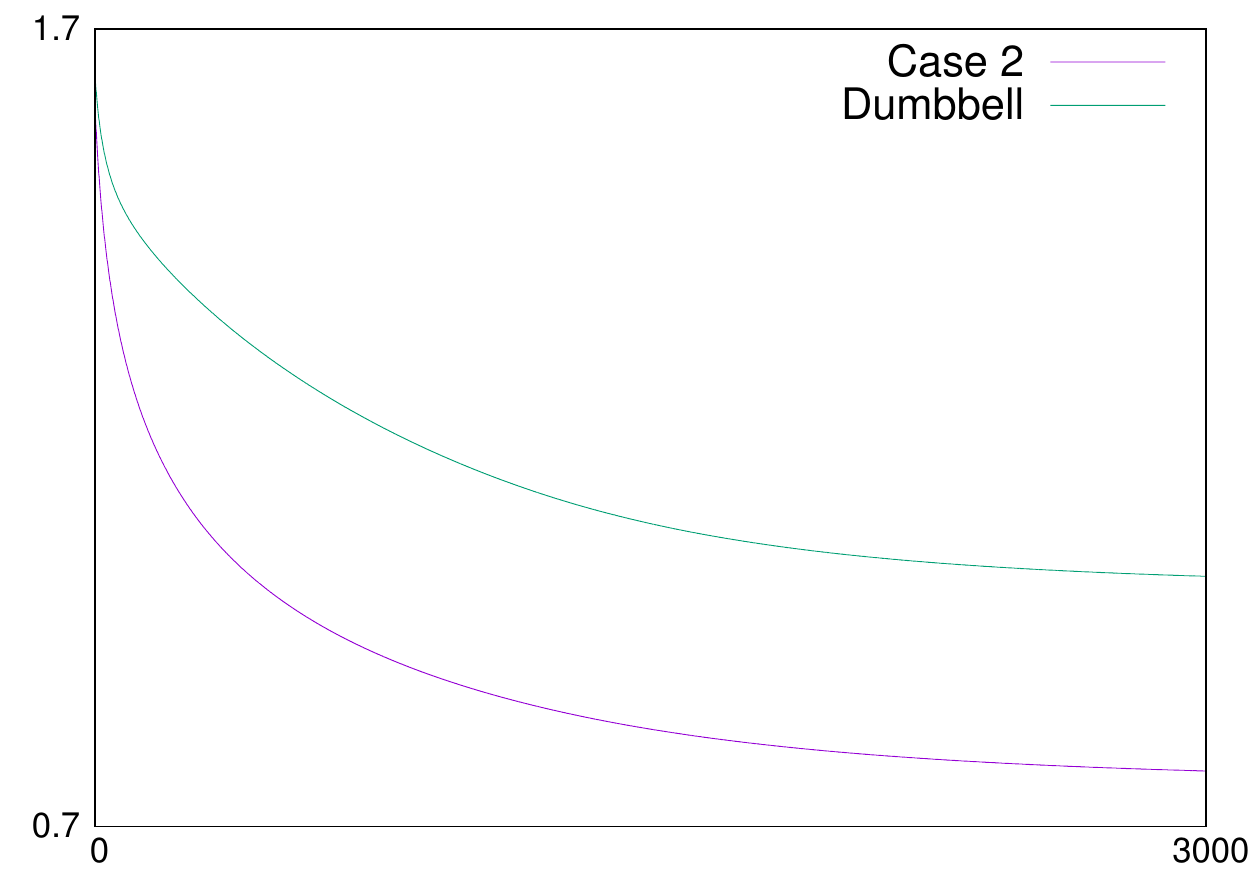}
      \end{center}
      \caption{The graph shows the energy decay with respect to the solution illustrated in Figure \ref{ncircase}. This graph is compared with the energy decay of Case 2 in Section \ref{sec41}.}
      \label{ncirenergy}
\end{figure}

\subsection{Channels with obstacles} \label{sec43}
In this section we set $\Omega$ as a channel and consider 2 different obstacles in it. The channel is given by $(-\infty, \infty) \times (0, 1.2)$ and the obstacles are displayed in Figures \ref{case2} and \ref{case3}. We set $W(x)=-0.002x_1$ as a driving force to move the particles past the obstacle. We further put $\Delta t=0.5$ and consider $n=900$ particles, which are initially put as a regular square grid in the range of $-1.7\le x_1\le 3.1$. We are interested in the manner in which particles may detach from the boundary after passing the obstacle, and to which extend the obstacle slows down the flow of the particles.

As a reference, we first solve Problem \ref{problem3} on the channel without obstacle. Figure \ref{case1} illustrates the dynamics.

\begin{figure}[htbp]
\centering
\subfloat{\includegraphics[width=0.48\textwidth]{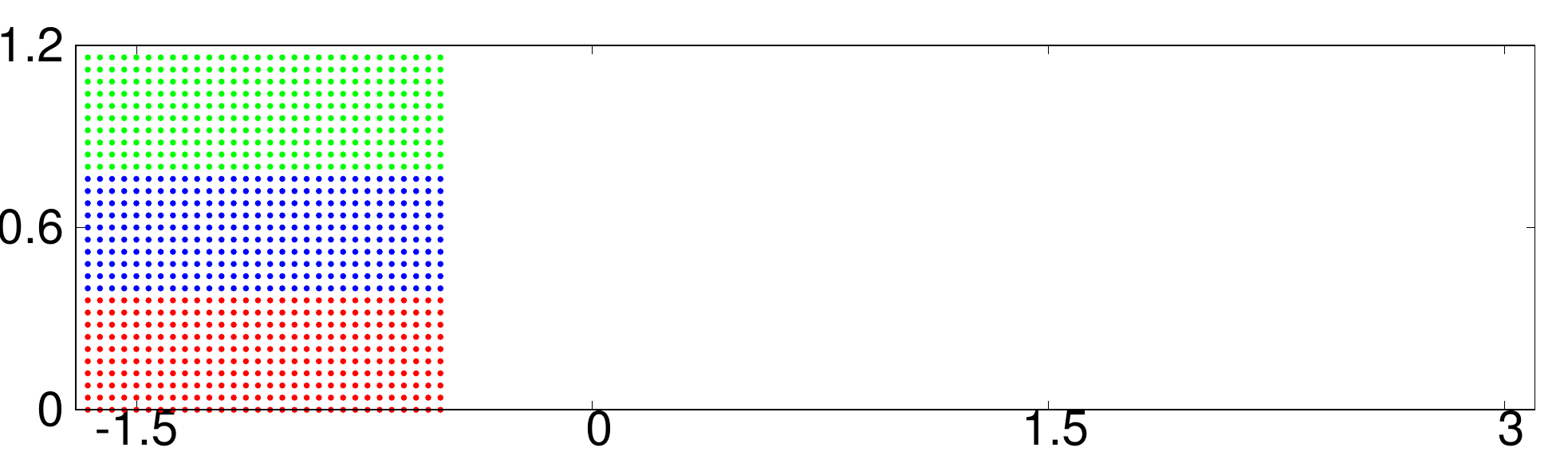}}
\subfloat{\includegraphics[width=0.48\textwidth]{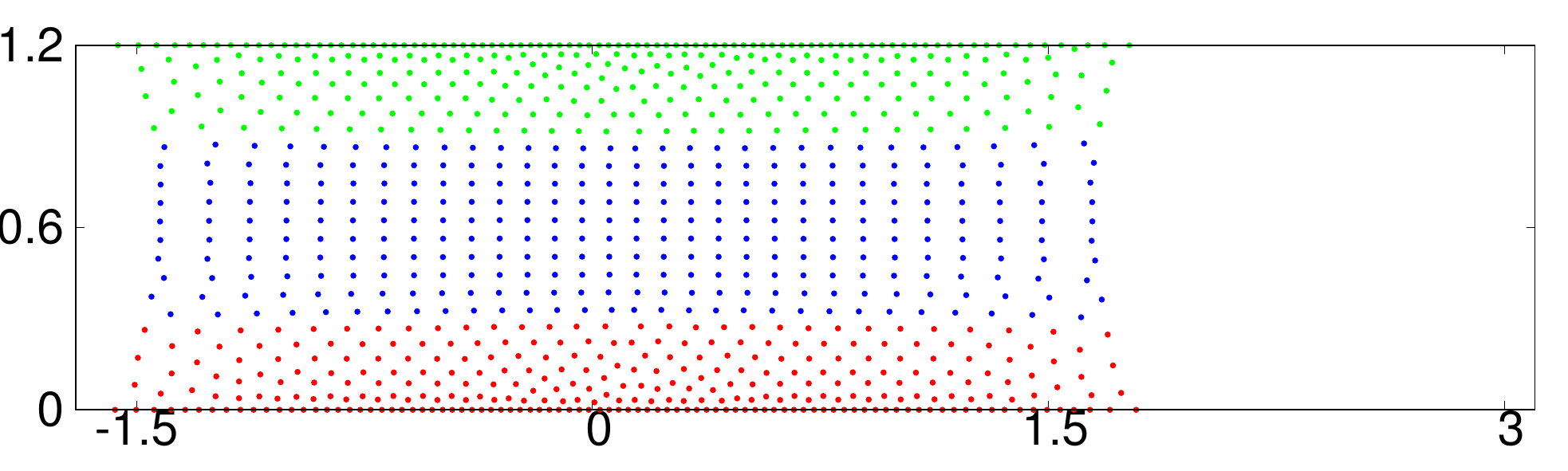}}\\
\subfloat{\includegraphics[width=0.48\textwidth]{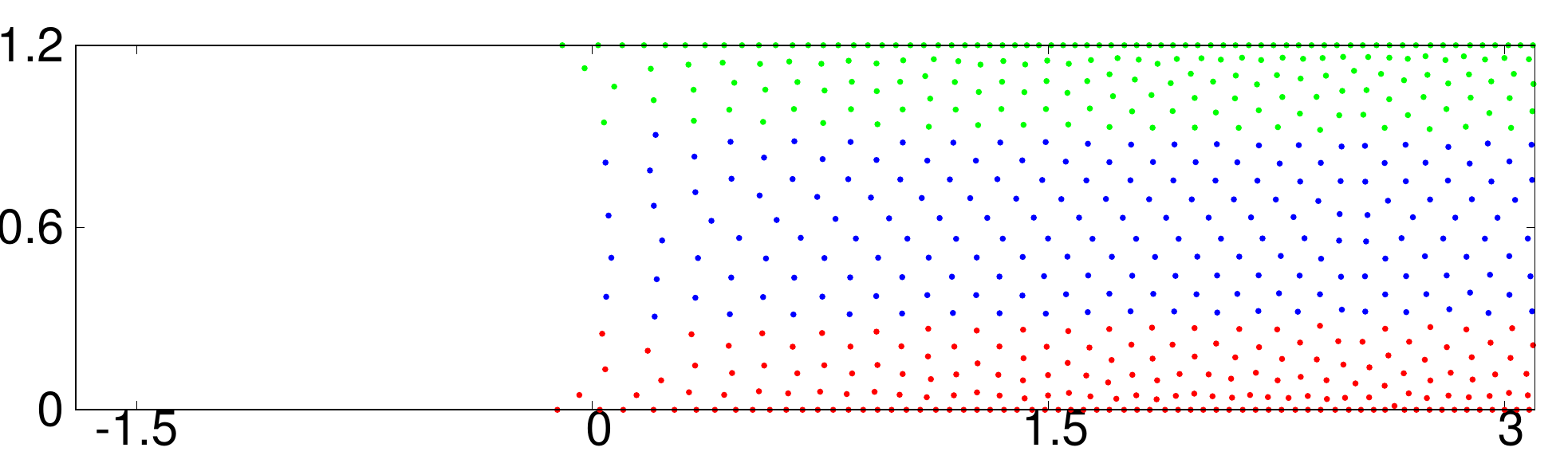}}
\subfloat{\includegraphics[width=0.48\textwidth]{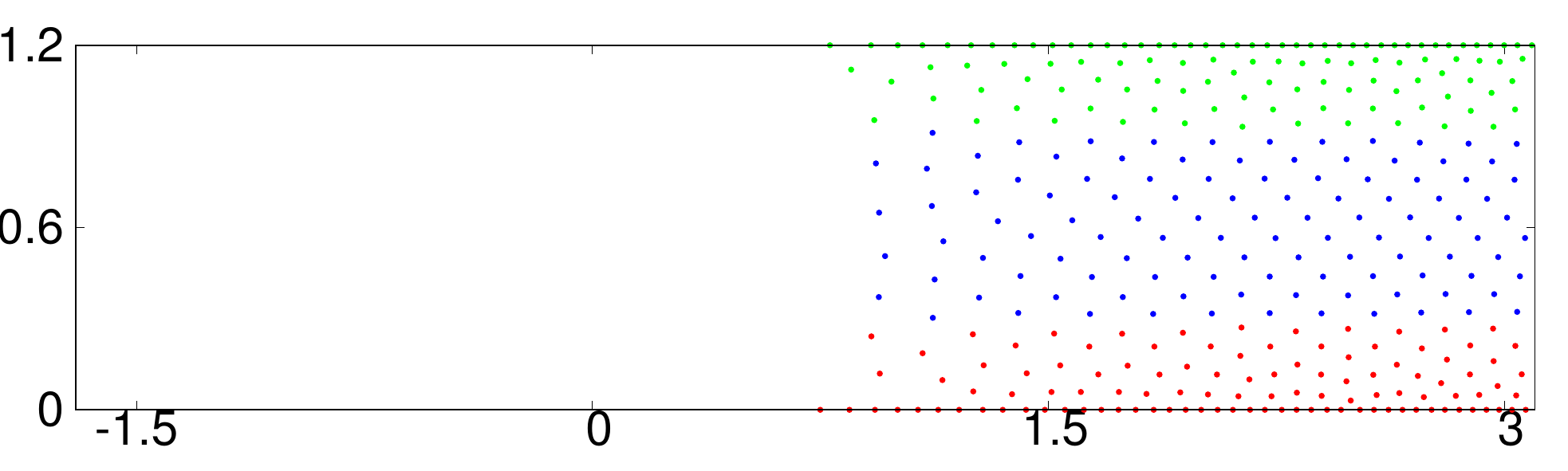}}
\caption{Channel without obstacle. The figures illustrate the particle positions at the four time instances $t=0$, $t=600$, $t=1800$ and $t=2400$ (from left to right, top to bottom). Again, the different colors for the particles are chosen solely for visualisation purposes.}
\label{case1}
\end{figure}

Next we insert a small obstacle, given by a bump, in the channel. Figure \ref{case2} illustrates the channel and the dynamics. We observe that the particles starting at the lower part of the channel pass the barrier more slowly than those that start in the upper part. The plots also show how some particles detach from the boundary after the obstacle. We also observe that the flow of the particles is hardly slowed down when compared to the dynamics displayed in Figure \ref{case1}.

\begin{figure}[htbp]
\centering
\subfloat{\includegraphics[width=0.48\textwidth]{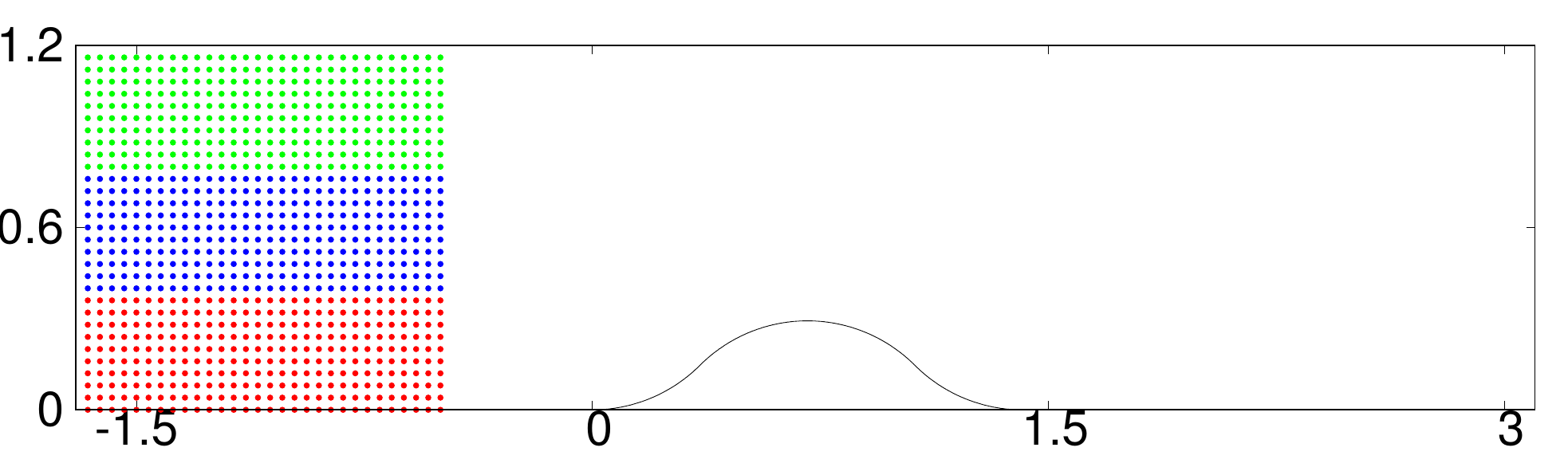}}
\subfloat{\includegraphics[width=0.48\textwidth]{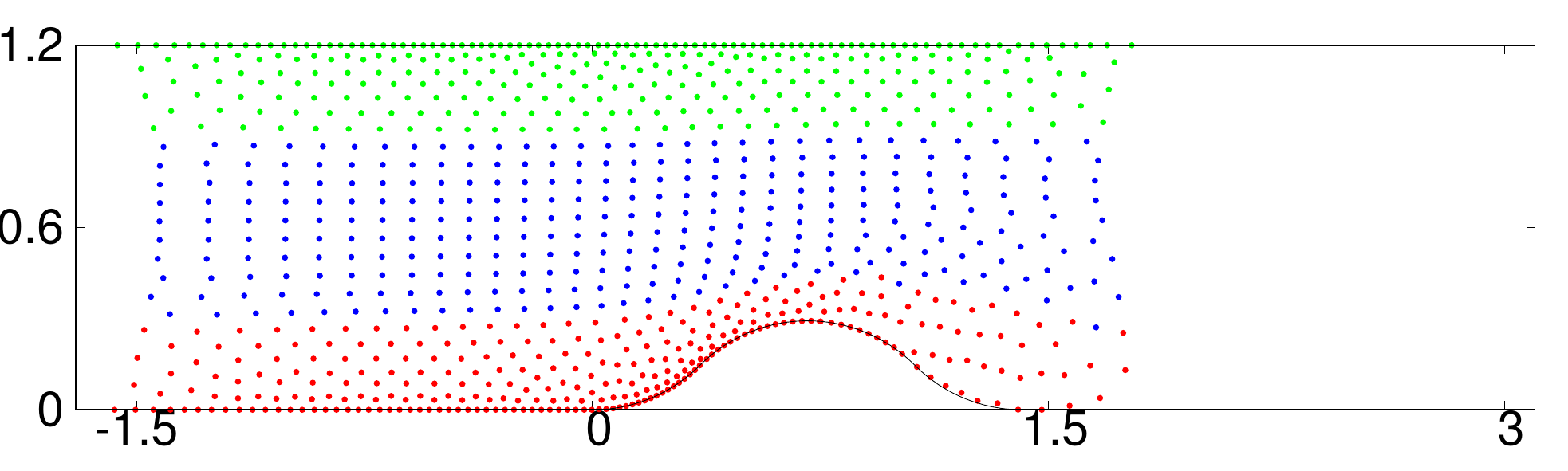}}\\
\subfloat{\includegraphics[width=0.48\textwidth]{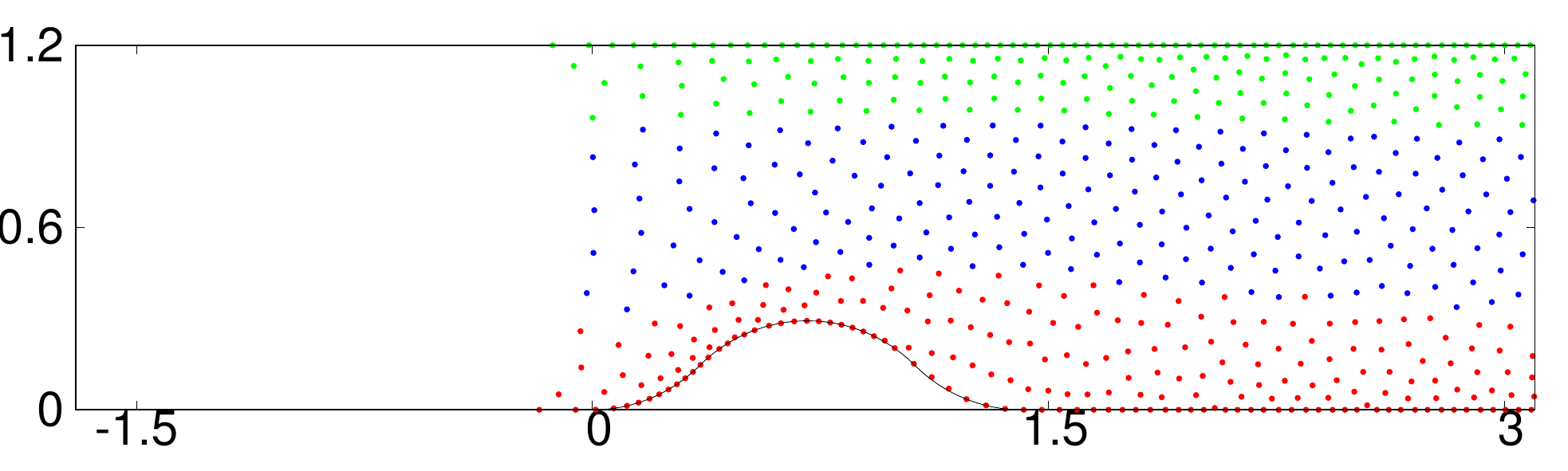}}
\subfloat{\includegraphics[width=0.48\textwidth]{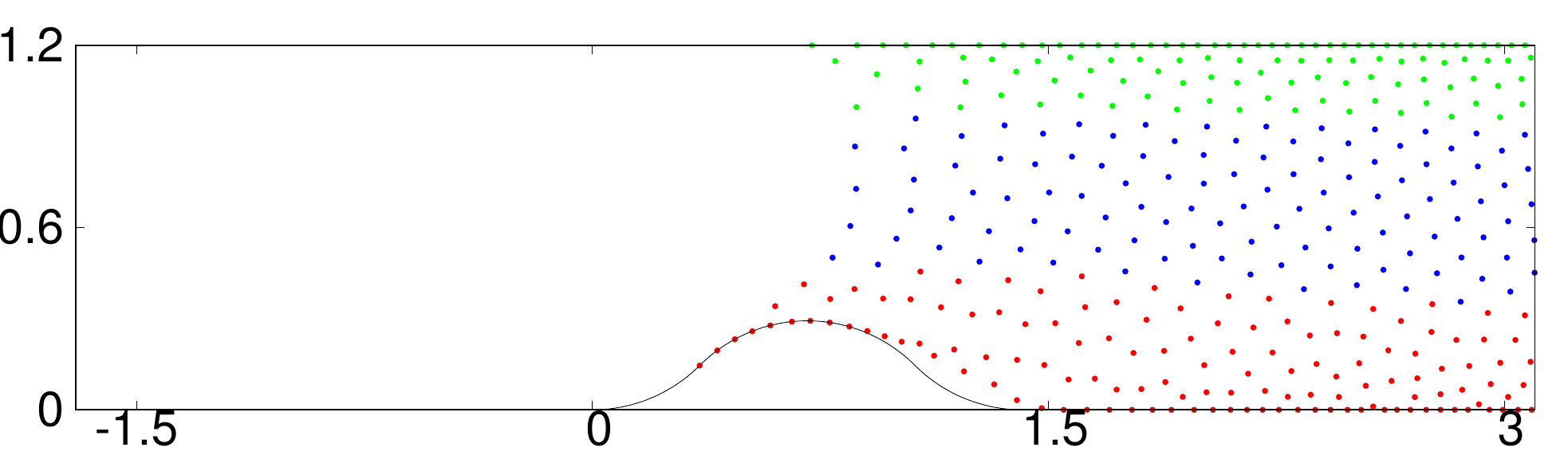}}
\caption{Channel with a bump. Other than changing the domain, the figures are constructed in the same was as in Figure \ref{case1}.}
\label{case2}
\end{figure}

Next we insert a large obstacle in the channel with a shape similar to a horse shoe. Figure \ref{case3} illustrates the dynamics. The detachment of the particles from the boundary of the horse shoe is clearly visible. We also observe that several particles get stuck in the left lower region around the obstacle. Also, the total flow of the particles through the channel per unit time is significantly less than the previous two cases. From the color coding we see that it is not so clear in which order the particles pass the obstacle. Also, after having passed the obstacle, the separation line between the red and blue particles becomes less clear, which indicates a small mixing effect.

\begin{figure}[htbp]
\centering
\subfloat{\includegraphics[width=0.48\textwidth]{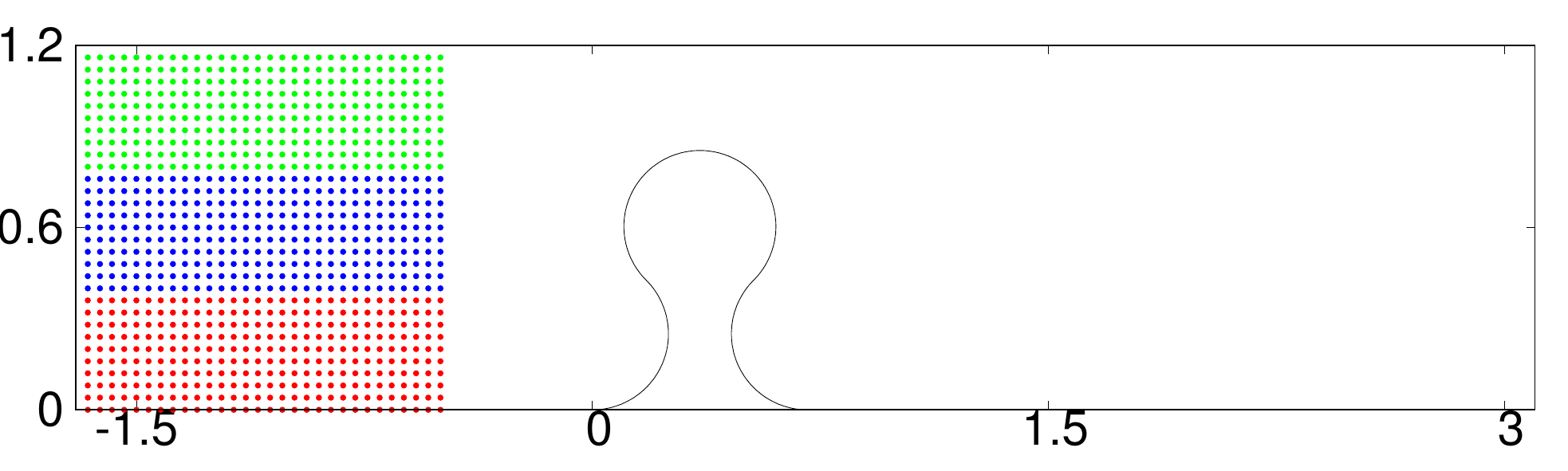}}
\subfloat{\includegraphics[width=0.48\textwidth]{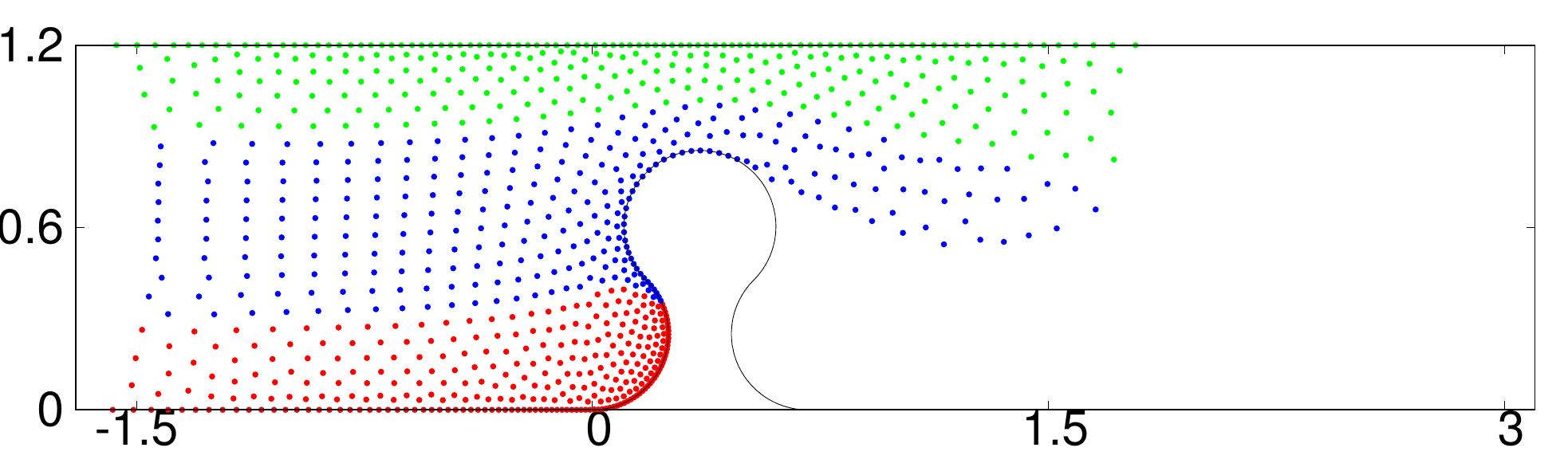}}\\
\subfloat{\includegraphics[width=0.48\textwidth]{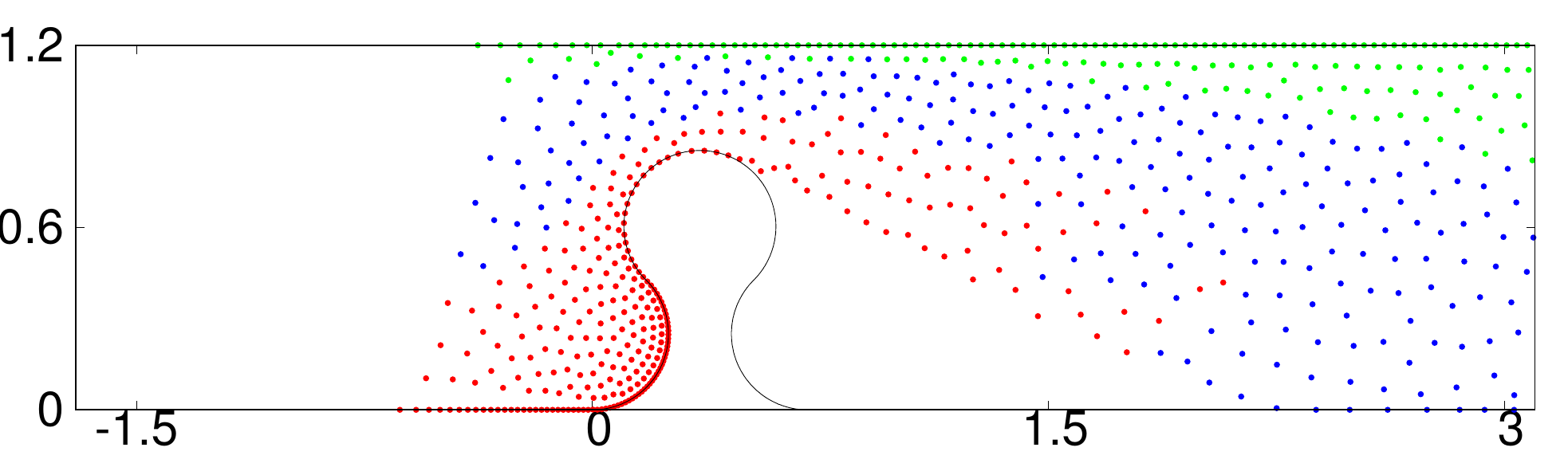}}
\subfloat{\includegraphics[width=0.48\textwidth]{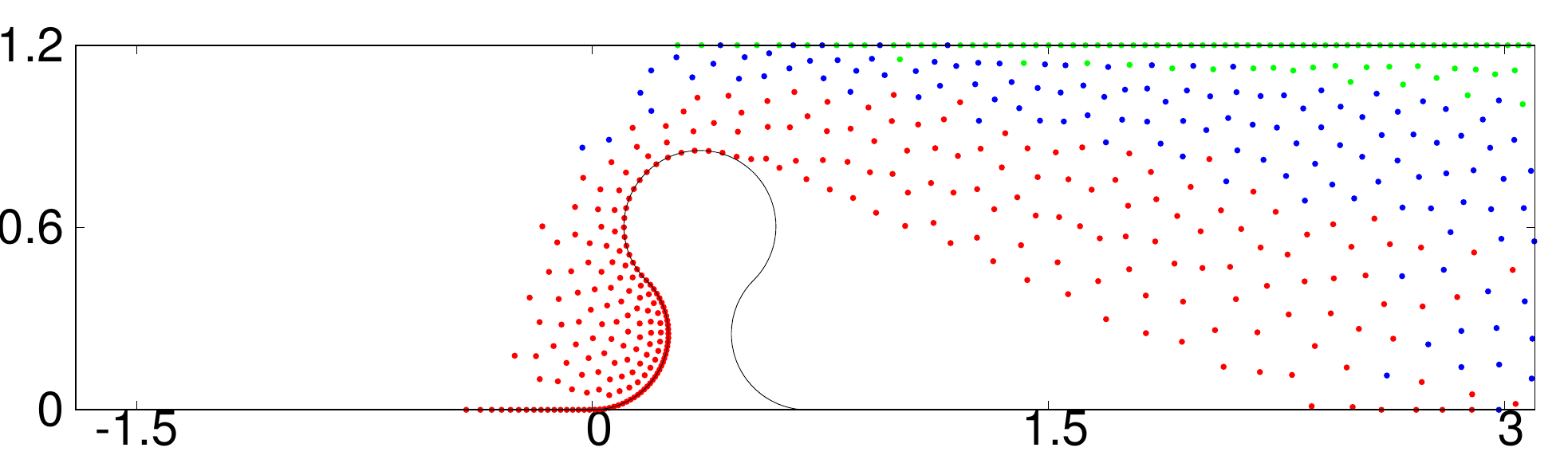}}
\caption{Channel with a horse-shoe shaped obstacle. Other than changing the domain, the figures are constructed in the same was as in Figure \ref{case1}.}
\label{case3}
\end{figure}

Figure \ref{energy} compares the energy decay of the dynamics between all three channels. While the small obstacle has little influence on the decay of the energy, the large obstacle clearly slows down the decay until most particles have passed the obstacle.
\begin{figure}[htbp] 
      \begin{center}
        \includegraphics[width=90mm]{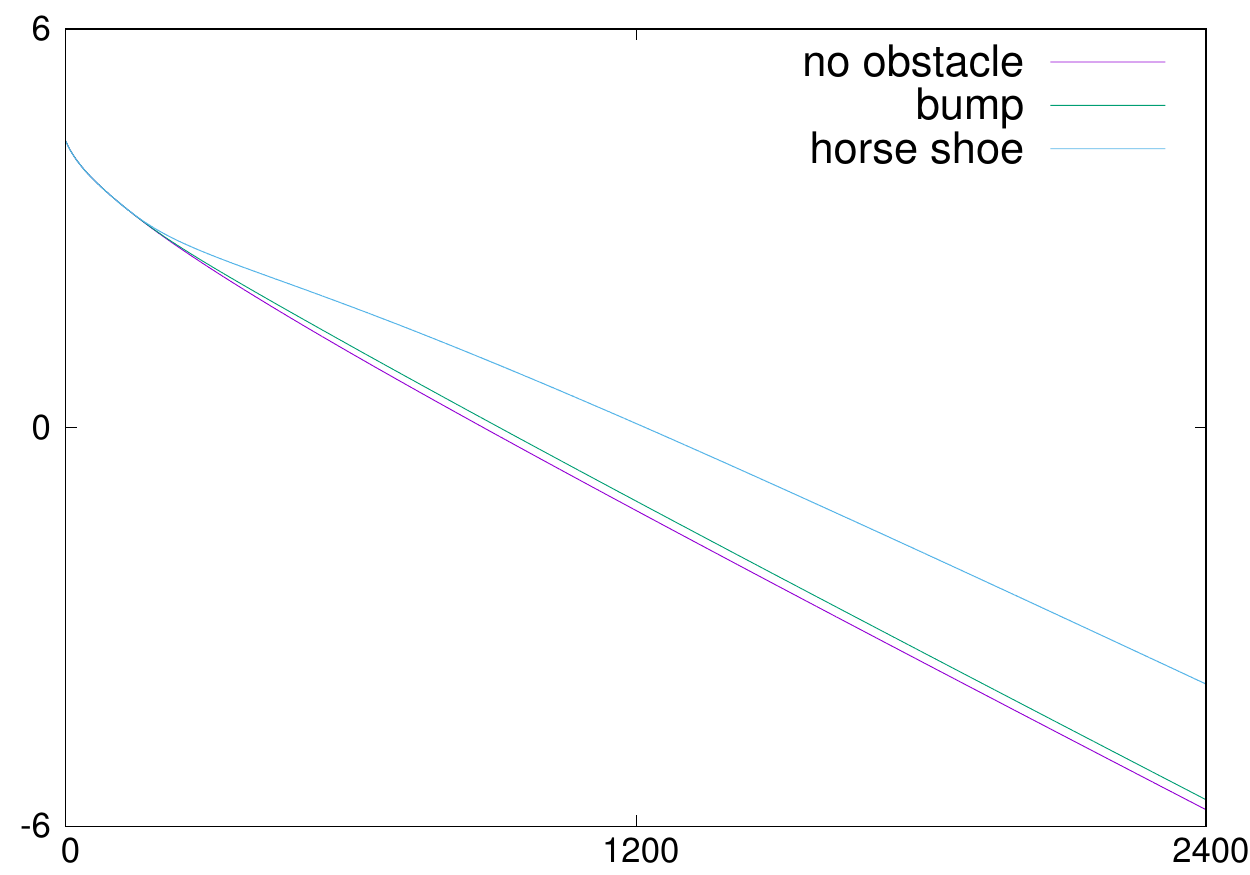}
      \end{center}
      \caption{Energy decay related to the particles dynamics in the three channels considered in Section \ref{sec43}. }
      \label{energy}
\end{figure}

\section{Conclusion and outlook}\label{sec5}
\setcounter{equation}{0}
In Theorem~\ref{theorem1} we proved existence and uniqueness of mild solutions to interacting particle systems when the particles are confined to a given domain (Problem \ref{problem1}). We also showed that when particles detach from the boundary, they do so with velocity tangential to the boundary (Proposition \ref{property}). We applied the results to gradient flows with confinement (see Theorem~\ref{Th3}). We have illustrated the dynamics by numerical simulations for various (non-convex) domains with a special focus on particles attaching to and detaching from the boundary.

The numerical illustrations spark several topics for future research. For instance, they suggest that the collective motion of the particles could be described by a continuum density. In future work we intend to obtain this density from Theorem \ref{theorem1} by passing to the many-particle limit $n \to \infty$ in Problem \ref{problem1}. Another example is that the coloring of the particles suggests that no mixing occurs (at least for convex domains); it would be interesting to formulate and prove a precise statement on the mixing of particles. Lastly, the group behavior of the particles shows a moving front in time for the support of the particle density; it would be interesting to describe this front with more precision.

\section*{Acknowledgments}\label{sec6}
\setcounter{equation}{0}
This work is partially supported by JSPS KAKENHI Grant Numbers JP16H02155 and JP17H02857.


\newcommand{\etalchar}[1]{$^{#1}$}

\end{document}